\documentclass[10pt]{article}       
\usepackage{amsmath,amssymb,graphicx,float,url,enumitem,xcolor,color,soul,subcaption,mathtools}
\usepackage{placeins}
\usepackage{amsthm}
\usepackage[colorlinks=true]{hyperref}

\newcommand{\norm}[1]{\left\lVert #1 \right\rVert}
\newcommand{\inner}[2]{\left\langle #1,#2 \right\rangle}
\newcommand{\abs}[1]{\left|#1\right|}
\newcommand{\normal}[1]{\frac{\partial #1}{\partial \nu}}
\newcommand{\conormal}[2]{\frac{\partial #1}{\partial \nu_{#2}}}

\renewcommand{\div}[0]{\textnormal{div}}
\newcommand{\curl}[0]{\textnormal{curl}}
\newcommand{\veccurl}[0]{\textbf{curl}}
\renewcommand{\Im}[0]{\textnormal{Im}}
\renewcommand{\Re}[0]{\textnormal{Re}}

\newcommand*\myat{{\fontfamily{ptm}\selectfont @}}

\renewenvironment{proof}{\noindent{\bfseries Proof.}}{\qed} 
\newtheorem{theorem}{Theorem}
\newtheorem{lemma}[theorem]{Lemma}

\newtheorem{corollary}[theorem]{Corollary}

\newtheorem{remark}[theorem]{Remark}

\newtheorem{assumption}[theorem]{Assumption}
\newtheorem{proposition}[theorem]{Proposition}

\begin{document}

\title{Asymptotic expansions of Stekloff eigenvalues for perturbations of inhomogeneous media
}


\author{Samuel Cogar\thanks{Department of Mathematics, Rutgers University, Piscataway, NJ 08854 (samuel.cogar\myat rutgers.edu}
}



\date{}
\providecommand{\keywords}[1]{\small{\textbf{Key words.} #1}}
\providecommand{\AMS}[1]{\small{\textbf{AMS subject classifications.} #1}}

\maketitle

\begin{abstract}

Eigenvalues arising in scattering theory have been envisioned as a potential source of target signatures in nondestructive testing of materials, whereby perturbations of the eigenvalues computed for a penetrable medium would be used to infer changes in its constitutive parameters relative to some reference values. We consider a recently introduced modification of the class of Stekloff eigenvalues, in which the inclusion of a smoothing operator guarantees that infinitely many eigenvalues exist under minimal assumptions on the medium, and we derive precise formulas that quantify the perturbation of a simple eigenvalue in terms of the coefficients of a perturbed inhomogeneous medium. These formulas rely on the theory of nonlinear eigenvalue approximation and regularity results for elliptic boundary-value problems with heterogeneous coefficients, the latter of which is shown to have a strong influence on the sensitivity of the eigenvalues corresponding to an anisotropic medium. A simple numerical example in two dimensions is used to verify the estimates and suggest future directions of study.

\end{abstract}

\begin{keywords}
inverse scattering, nondestructive testing, non-selfadjoint eigenvalue problems, Laplace-Beltrami operator, nonlinear eigenvalue problems
\end{keywords}

\begin{AMS}
35J25, 35P05, 35P25, 35R30
\end{AMS}

\section{Introduction}
\label{sec:intro}

A problem of interest in inverse scattering theory is to detect changes in a penetrable medium from a knowledge of how waves are scattered from it, serving as an example of nondestructive testing of materials. A recent approach is to use certain classes of eigenvalues arising in scattering theory as target signatures, in which perturbations of the eigenvalues may potentially be used to infer changes in the medium relative to some reference configuration. This approach was initially envisioned for the class of transmission eigenvalues (cf. \cite{cakoni_colton_haddar}); however, while this avenue of research has been fruitful from a mathematical perspective, difficulties have arisen concerning their practical implementation. In particular, measured scattering data may only be used to detect real transmission eigenvalues, none of which exist for an absorbing medium. Moreover, the detection of transmission eigenvalues requires the collection of multifrequency data in a predetermined range, thereby providing little freedom in the choice of interrogating frequency. \par

By following the same reasoning that led to the development of transmission eigenvalues for a penetrable medium, a new class of eigenvalues was introduced in \cite{cakoni_colton_meng_monk} by considering a modification of the measured acoustic scattering data obtained by subtracting from it the scattering data for an exterior impedance problem dependent upon a parameter $\lambda$. This specific case resulted in the class of Stekloff eigenvalues, and the appearance of the eigenparameter $\lambda$ in the so-called auxiliary problem rather than the physical problem removed many of the aforementioned restrictions that complicate the practical application of transmission eigenvalues: eigenvalues anywhere in the complex plane can be computed from measured scattering data, and a single interrogation frequency may be chosen at will. Concerning the latter point, the collection of scattering data for multiple choices of the frequency is essentially replaced by the computation of auxiliary scattering data for multiple choices of the eigenparameter $\lambda$. \par

From this approach was born various new eigenvalue problems (cf. \cite{audibert_cakoni_haddar,camano_lackner_monk,cogar,cogar_thesis,cogar2020,cogar2020EM,cogar_colton_meng_monk,cogar_colton_monk,cogar_monk,halla2,halla1,li}),
which share many qualities but often require different techniques in the course of their analysis. Two properties of these classes of eigenvalues are of special interest in a practical setting but difficult to study: the existence of eigenvalues and their sensitivity to changes in the medium to which they correspond. The former property is necessary for eigenvalues to find use as a target signature. It is generally a simple matter to establish for a nonabsorbing penetrable medium, as the resulting eigenvalue problem is often self-adjoint; however, in the absorbing case it has been shown only under stringent smoothness requirements on the medium for some classes of eigenvalues (cf. \cite{cakoni_colton_meng_monk,cogar_colton_meng_monk}) using Agmon's theory of non-selfadjoint boundary value problems (cf. \cite{agmon}), but existence results are not available otherwise. This issue was addressed in \cite{cogar2020} by introducing a smoothing operator in the construction of Stekloff eigenvalues, leading to the class of $\delta$-Stekloff eigenvalues, for which it was shown that infinitely many eigenvalues exist--even for an absorbing medium--provided that the degree of smoothing is sufficiently high. It is for this modified class of Stekloff eigenvalues (and consequently the standard class by choosing the smoothing parameter $\delta=0$) that we address the second issue of sensitivity. Knowledge of the relationship between the eigenvalues and changes in the medium is of vital importance in the intended application of nondestructive testing, as the perturbations of the eigenvalues are the only information that is obtained when using eigenvalues as a target signature. Thus, precise results in this direction may improve the applicability of using eigenvalues as target signatures in a realistic setting. \par

While the sensitivity of eigenvalues was also discussed in \cite{cogar2020}, the main effort was to provide stability estimates for a certain solution operator that was related to the class of $\delta$-Stekloff eigenvalues, and the perturbations of the eigenvalues themselves were not quantified. The present aim is to improve upon this idea in three ways by reformulating the $\delta$-Stekloff eigenvalue problem as a nonlinear eigenvalue problem and using the asymptotic correction formula given in \cite{moskow} (see \cite{cakoni_moskow,cakoni_moskow_rome} for a similar approach to transmission eigenvalues). First, we obtain an asymptotic formula for the perturbation of a simple $\delta$-Stekloff eigenvalue due to a change in the consitutitive parameters of the medium. Second, we use this formula to arrive at upper bounds relative to the norm of the perturbed coefficients. Third, we extend our study of eigenvalue perturbations to the case of anisotropic media, which necessitates precise results on the regularity of the medium and significantly increases the applicability of our findings. \par

The remainder of our discussion is outlined as follows. As we will frequently use a few key results throughout our investigation, Section \ref{sec:prelim} is devoted to their statements and the basic notation on which they rely. In Section \ref{sec:eig} we introduce the $\delta$-Stekloff eigenvalue problem that we will consider and provide some basic results from \cite{cogar2020}, including the aforementioned result that infinitely many eigenvalues exist when the smoothing parameter $\delta$ is sufficiently large relative to the dimension. We reformulate the $\delta$-Stekloff eigenvalue problem as a nonlinear eigenvalue problem in Section \ref{sec:perturbation} and use the first-order correction formula from \cite{moskow} to arrive at our main result, which quantifies the sensitivity of the eigenvalues to changes in the medium. In Section \ref{sec:isotropic} we briefly consider the simpler case in which the material is isotropic, allowing us to leverage higher regularity results in order to improve the upper bound on the perturbations of the eigenvalues (in the sense of achieving a lower $L^p$-norm of the perturbed coefficients). After comparing the results of a simple numerical example in Section \ref{sec:numerics} to the theoretical results of the preceding section, we conclude in Section \ref{sec:conclusion} with an overview of the impact and potential use of the main results along with some unanswered questions.

\section{Preliminaries and notation}
\label{sec:prelim}

We devote this section to introducing the notation and relevant results that will be used throughout our investigation. We consider a domain $B\subseteq\mathbb{R}^d$, $d=2,3$, which we assume to have connected exterior $\mathbb{R}^d\setminus\overline{B}$ and smooth, simply-connected boundary $\partial B$, the latter of which has a unit outward normal vector denoted by $\nu$. On this domain we define the standard $L^p$-spaces for $p\in[1,\infty)$ by
\begin{equation*}
L^p(B) = \left\{ f:B\to\mathbb{C} \;\middle|\; f \text{ is measurable and } \int_B \abs{f}^p \,dx < \infty \right\}
\end{equation*}
and equip each space with the norm
\begin{equation*}
\norm{f}_{L^p(B)} = \left( \int_B \abs{f}^p \,dx \right)^{1/p}.
\end{equation*}
In particular, we recall that $L^2(B)$ is a Hilbert space, and we denote its inner product by $(\cdot,\cdot)_B$, which is defined as
\begin{equation*}
(f_1,f_2)_B = \int_B f_1 \overline{f_2} \,dx \quad\forall f_1,f_2\in L^2(B).
\end{equation*}
We define the space $L^\infty(B)$ by
\begin{equation*}
L^\infty(B) = \left\{ f:B\to\mathbb{C} \;\middle|\; f \text{ is measurable and } \text{ess supp}(\abs{f}) < \infty \right\}
\end{equation*}
and equip it with the norm
\begin{equation*}
\norm{f}_{L^\infty(B)} = \text{ess supp}(\abs{f}).
\end{equation*}
For any positive integer $m>0$, we define the Sobolev space $H^m(B)$ in the usual way (cf. \cite{adams,brezis}) as
\begin{equation*}
H^m(B) = \{f\in L^2(B) \mid D^\alpha f\in L^2(B), \;\abs{\alpha}\le m\},
\end{equation*}
where $D^\alpha$ refers to the distributional derivative corresponding to the multi-index $\alpha$. For $t\in(0,1)$, the fractional Sobolev space $H^t(B)$ is defined by means of interpolation between $H^1(B)$ and $L^2(B)$ (cf. \cite{tartar}), and we define the space $H^{t+1}(B)$ as
\begin{equation*}
H^{t+1}(B) = \{f\in H^1(B) \mid \nabla f\in(H^t(B))^d\}
\end{equation*}
with norm $\norm{f}_{H^{t+1}(B)}^2 = \norm{f}_{H^1(B)}^2 + \norm{\nabla f}_{(H^t(B))^d}^2$ (cf. \cite{bonito_guermond_luddens}). We denote by $\tilde{H}^{-t}(B)$ the corresponding dual space of $H^t(B)$, defined in terms of a Gelfand triple with pivot space $L^2(B)$. \par
If $X$ denotes any of these spaces, we use the notation $(X)^d$ and $(X)^{d\times d}$ to represent spaces of $d$-dimensional vectors and $d\times d$ matrices, respectively, whose entries lie in $X$, and we define the respective norms by
\begin{equation*}
\norm{(u_1,\dots,u_d)^T}_{(X)^d} = \sum_{i=1}^d \norm{u_i}_X, \quad \norm{\left(\begin{array}{ccc} u_{11} & \cdots & u_{1d} \\ \vdots & \ddots & \vdots \\ u_{d1} & \cdots & u_{dd} \end{array}\right)}_{(X)^{d\times d}} = \sum_{i,j=1}^d \norm{u_{ij}}_X.
\end{equation*}
We also define the Sobolev space $H^t(\partial B)$ for $t>0$ as the trace space of $H^{t+\frac{1}{2}}(B)$ with the image norm (cf. \cite{adams}), and we define the dual space $H^{-t}(\partial B)$ with $L^2(\partial B)$ as the pivot space, where the inner product $\inner{\cdot}{\cdot}_{\partial B}$ on $L^2(\partial B)$ defined by
\begin{equation*}
\inner{f_1}{f_2}_{\partial B} = \int_{\partial B} f_1 \overline{f_2} \,ds \quad\forall f_1,f_2\in L^2(\partial B)
\end{equation*}
will also be used to denote the duality pairing of $H^{-t}(\partial B)$ and $H^t(\partial B)$. 

We state a special case of the Sobolev embedding theorem (cf. \cite{adams}), recalling that $B$ is assumed to be a smooth domain in $\mathbb{R}^d$.

\begin{theorem} \label{theorem:sobolev}

\textnormal{\textbf{(Sobolev Embedding Theorem)}} Suppose that $f\in H^t(B)$ for some $t>0$.
\begin{enumerate}[label=(\roman*)]
\item If $2t>d$, then $f\in C_b(B)$, the space of bounded continuous functions on $B$ equipped with the uniform norm, with the estimate
\begin{equation*}
\norm{f}_{C_b(B)} \le C_t \norm{f}_{H^t(B)}.
\end{equation*}
\item If $2t=d$, then $f\in L^p(B)$ for every $p\in[2,\infty)$ with the estimate
\begin{equation*}
\norm{f}_{L^p(B)} \le C_{t,p} \norm{f}_{H^t(B)}.
\end{equation*}
\item If $2t<d$, then $f\in L^{2d/(d-2t)}(B)$ with the estimate
\begin{equation*}
\norm{f}_{L^{2d/(d-2t)}(B)} \le C_t \norm{f}_{H^t(B)}.
\end{equation*}
\end{enumerate}
Each constant is independent of $f$.

\end{theorem}

We will also use this result for vector-valued functions, in which case we apply it to each component individually. The following result from \cite{brezis} is a specific form of H{\"o}lder's inequality involving the product of three functions lying in suitable $L^p$-spaces. We will frequently apply the resulting estimate with $q$ and $r$ determined by the Sobolev embedding theorem and $p$ determined by the condition in the theorem.

\begin{theorem} \label{theorem:holder}

\textnormal{\textbf{(Three-Term H{\"o}lder's Inequality)}} If $f_1\in L^p(B)$, $f_2\in L^q(B)$, and $f_3\in L^r(B)$ for $p,q,r\in[1,\infty]$ satisfying
\begin{equation*}
\frac{1}{p} + \frac{1}{q} + \frac{1}{r} = 1,
\end{equation*}
then $f_1 f_2 f_3\in L^1(B)$ and 
\begin{equation} \label{holder_interp}
\norm{f_1 f_2 f_3}_{L^1(B)} \le \norm{f_1}_{L^p(B)} \norm{f_2}_{L^q(B)} \norm{f_3}_{L^r(B)}.
\end{equation}

\end{theorem}

We remark that the three-term H{\"o}lder's inequality may be applied when $f_1 = \mathbf{F}_1$ is a matrix-valued function and $f_2 = \mathbf{f}_2$ and $f_3 = \mathbf{f}_3$ are vector-valued functions (with the integrand replaced by $\mathbf{F}_1\mathbf{f}_2\cdot\mathbf{f}_3$). While the exact implementation of the result in this case requires expanding the matrix multiplication and evaluating the dot product in the integrand, the resulting inequality is identical in form to \eqref{holder_interp}, and consequently we use it without further comment from this point onward.

Finally, we state the following interpolation inequality from \cite{brezis}, which can be seen as a consequence of the standard H{\"o}lder's inequality, to estimate the remainder terms in our asymptotic formulas. Note that this statement is more specific than that found in \cite{brezis}, covering only the case of a bounded domain.

\begin{theorem} \label{theorem:holder_interp}

\textnormal{\textbf{(H{\"o}lder's Interpolation Inequality)}} Assume that $f\in L^q(B)$ for some $q\in[1,\infty]$ and that $p\in[1,q]$. If $r\in[p,q]$ and $\alpha\in[0,1]$ satisfy
\begin{equation*}
\frac{1}{r} = \frac{\alpha}{p} + \frac{1-\alpha}{q},
\end{equation*}
then
\begin{equation} \label{holder_interp}
\norm{f}_{L^r(B)} \le \norm{f}_{L^p(B)}^\alpha \norm{f}_{L^q(B)}^{1-\alpha}.
\end{equation}

\end{theorem}

For both Theorems \ref{theorem:holder} and \ref{theorem:holder_interp}, we adopt the convention that $\frac{1}{\infty} = 0$ in the case that one of the functions lies in $L^\infty(B)$.

\section{The $\delta$-Stekloff eigenvalue problem}
\label{sec:eig}

We recall that the domain $B$ introduced in the previous section has smooth, simply-connected boundary $\partial B$. Before we present the eigenvalue problem of interest from \cite{cogar2020}, we must introduce a smoothing operator that will be necessary for certain existence results.  Denoting by $\nabla_{\partial B}$, $\div_{\partial B}$, $\curl_{\partial B}$, and $\veccurl_{\partial B}$ the surface gradient, surface divergence, scalar surface curl, and vector surface curl, respectively, we define the nonnegative Laplace-Beltrami operator as
\begin{equation*}
\Delta_{\partial B} := -\div_{\partial B}\nabla_{\partial B} = \curl_{\partial B} \veccurl_{\partial B}.
\end{equation*}
We refer to \cite{monk,nedelec} for details on these surface differential operators, and we note that we have introduced a negative sign as in \cite{jost} in order to ensure nonnegativity of $\Delta_{\partial B}$. We state the following results from \cite{sayas_brown_hassell} on the spectral properties of this operator.

\begin{proposition} \label{prop:LB}

There exists an orthonormal basis $\{Y_m\}_{m=0}^\infty$ of $L^2(\partial B)$ and a nondecreasing divergent sequence of nonnegative real numbers $\{\mu_m\}_{m=0}^\infty$, counted according to multiplicity, such that
\begin{equation*}
\Delta_{\partial B} Y_m = \mu_m Y_m, \; m\ge0.
\end{equation*}
The first eigenvalue is $\mu_0 = 0$ with constant eigenfunction $Y_0 = \abs{\partial B}^{-1/2}$, and $\mu_m>0$ for $m\ge1$.

\end{proposition}

If $B$ is chosen to be a ball in $\mathbb{R}^d$, then the eigenfunctions $\{Y_m\}$ are given by complex exponentials when $d=2$ and spherical harmonics when $d=3$ (cf. \cite{nedelec}). By representing a given $\xi\in L^2(\partial B)$ (or, more generally, any surface distribution on $\partial B$) in the form
\begin{equation} \label{SHexp}
\xi = \sum_{m=0}^\infty \xi_m Y_m, \; \xi_m := \inner{\xi}{Y_m}_{\partial B},
\end{equation}
it can be shown (cf. \cite{nedelec}) that the Sobolev space $H^t(\partial B)$, $t\ge0$, may be characterized as
\begin{equation*}
H^t(\partial B) = \left\{ \xi\in L^2(\partial B) \;\middle|\; \sum_{m=0}^\infty (1+\mu_m)^t \abs{\xi_m}^2 < \infty \right\}
\end{equation*}
with equivalent norm
\begin{equation*}
\norm{\xi}_{H^t(\partial B)}^2 := \sum_{m=0}^\infty (1+\mu_m)^t \abs{\xi_m}^2.
\end{equation*}
We use this characterization of Sobolev spaces on $\partial B$ to define the \emph{Bessel potential operator} $S_\delta := (I+\Delta_{\partial B})^{-\delta}$ for a real number $\delta$, which may be expressed as
\begin{equation*}
S_\delta \xi = \sum_{m=0}^\infty (1+\mu_m)^{-\delta} \xi_m Y_m.
\end{equation*}

We now proceed to introduce the $\delta$-Stekloff eigenvalue problem from \cite{cogar2020}, which depends upon the constitutive parameters $A$ and $n$ of an inhomogeneous medium that is supported in a bounded Lipschitz domain $D\subseteq\mathbb{R}^d$ with connected exterior. The matrix-valued function $A$ represents the anisotropic properties of the medium, and we assume that $A(x)$ is Hermitian and positive-definite for each $x\in\mathbb{R}^d$. The index of refraction $n$ is assumed to lie in $L^\infty(\mathbb{R}^d)$ such that $\Re(n)\ge\alpha>0$ and $\Im(n)\ge0$ a.e. in $D$. We assume that $D$ is compactly contained within $B$, implying that $A=I$ and $n=1$ in the open set $B\setminus\overline{D}$. We adopt one final assumption to ensure regularity results that will be necessary later (cf. \cite{bonito_guermond_luddens}).

\begin{assumption} \label{assumption:A}

We assume that $B$ admits a partition $\{\Omega_m\}_{m=1}^M$ with interface $\Sigma$ for which $A$ lies in the space
\begin{equation*}
W_\Sigma^{1,\infty}(B) := \left\{\tilde{A}\in(L^\infty(B))^{d\times d} \;\left|\; \begin{array}{cc} \nabla(\tilde{a}_{ij}|_{\Omega_m})\in(L^\infty(\Omega_m))^d \\ 1\le i,j\le d, \; 1\le m\le M \end{array}\right.\right\},
\end{equation*}
where $\tilde{a}_{ij}$ denotes the corresponding entry of the matrix-valued function $\tilde{A}$.

\end{assumption}

We equip the space $W_\Sigma^{1,\infty}(B)$ introduced in Assumption \ref{assumption:A} with the natural norm
\begin{equation*}
\norm{\tilde{A}}_{W_\Sigma^{1,\infty}(B)} = \norm{\tilde{A}}_{(L^\infty(B))^{d\times d}} + \sum_{m=1}^M \sum_{i,j=1}^d \norm{\nabla(\tilde{a}_{ij}|_{\Omega_m})}_{(L^\infty(\Omega_m))^d}.
\end{equation*}

For a given $\delta\ge0$, we consider the \emph{$\delta$-Stekloff eigenvalue problem} in which we seek $\lambda\in\mathbb{C}$ and a nonzero $u\in H^1(B)$ satisfying
\begin{subequations} \label{deltastek}
\begin{align}
\nabla\cdot A\nabla u + k^2 n u &= 0 \text{ in } B, \label{deltastek1} \\
\conormal{u}{A} + \lambda S_\delta u &= 0 \text{ on } \partial B, \label{deltastek2}
\end{align}
\end{subequations}
where $k>0$ is a fixed wave number resulting from the chosen frequency of incident wave and $\conormal{u}{A} = \nu\cdot A\nabla u$ is the conormal derivative. We refer to \cite{cogar2020} for a derivation of \eqref{deltastek} along with a discussion of its relationship to the problem of acoustic scattering by an inhomogeneous medium. We call a value of $\lambda$ for which \eqref{deltastek} has a nontrivial solution a \emph{$\delta$-Stekloff eigenvalue}, and we remark that with $\delta=0$ we recover the standard class of Stekloff eigenvalues introduced in \cite{cakoni_colton_meng_monk} (see also \cite{audibert_cakoni_haddar}). This problem with $A=I$ was introduced and studied in \cite{cogar2020}, where the primary aim was to use Lidski's theorem (cf. \cite{ringrose}) to prove that infinitely many eigenvalues exist even when $n$ is generally complex-valued. Since we have assumed that $A=I$ and $n=1$ in $B\setminus\overline{D}$, the same approach may be applied to the anisotropic problem to obtain the following result.

\begin{theorem} \label{theorem:lidski}

Assume that there exists no nontrivial $w\in H^1(B)$ satisfying
\begin{subequations} \label{dir}
\begin{align}
\nabla\cdot A\nabla w + k^2 n w &= 0 \text{ in } B, \label{dir1} \\
w &= 0 \text{ on } \partial B. \label{dir2}
\end{align}
\end{subequations}
If $\delta>\frac{d}{2}-1$, then the $\delta$-Stekloff eigenvalues form an infinite discrete set without finite accumulation point.

\end{theorem}

With the eigenvalue problem of interest formulated and basic existence results in hand, we proceed in the next section to reformulate \eqref{deltastek} as a nonlinear eigenvalue problem in order to investigate the sensitivity of the $\delta$-Stekloff eigenvalues to changes in $A$ and $n$. We remark that Assumption \ref{assumption:A} is essential to our investigation, as it is sufficient to obtain higher regularity of solutions to \eqref{deltastek} and related nonhomogeoneous problems that will influence the sensitivity of the eigenvalues.


\section{Perturbation estimates}
\label{sec:perturbation}

We consider a sequence $\{(A_h,n_h)\}_{h\ge0}$ of coefficient pairs that satisfy the assumptions for $(A,n)$ given in Section \ref{sec:eig}, where in particular Assumption \ref{assumption:A} is satisfied by all $A_h$ with the same partition. The pair $(A_0,n_0)$ corresponds to the reference medium, of which each pair $(A_h,n_h)$ with $h>0$ is viewed as a perturbation. We also redefine the inner product on $H^1(B)$ as
\begin{equation*}
(u,u')_{H^1(B)} = (A_0\nabla u,\nabla u')_B + k^2(u,u')_B \quad\forall u,u'\in H^1(B)
\end{equation*}
and consider the corresponding induced norm $\norm{\cdot}_{H^1(B)}$, which are equivalent to the standard definitions due to our assumptions on $A_0$. We begin by reformulating \eqref{deltastek} with $(A,n) = (A_h,n_h)$ as a nonlinear eigenvalue problem.

We define the operators $\mathbb{A}_h, \mathbb{K}_h, \mathbb{B}:H^1(B)\to H^1(B)$ by means of the Riesz representation theorem such that
\begin{align*}
(\mathbb{A}_h u,u')_{H^1(B)} &= (A_h\nabla u,\nabla u')_B + k^2(u,u')_B, \\
(\mathbb{K}_h u,u')_{H^1(B)} &= -k^2((n_h+1)u,u')_B, \\
(\mathbb{B} u,u')_{H^1(B)} &= -\inner{S_\delta u}{u'}_{\partial B}
\end{align*}
for all $u,u'\in H^1(B)$, and we observe that \eqref{deltastek} is equivalent to finding $\lambda\in\mathbb{C}$ and a nonzero $u\in H^1(B)$ for which
\begin{equation} \label{dsop}
(\mathbb{A}_h + \mathbb{K}_h)u = \lambda\mathbb{B}u.
\end{equation}
Note that each $\mathbb{A}_h$ is invertible and that, due to our redefinition of the inner product on $H^1(B)$, $\mathbb{A}_0$ is equal to the identity operator. By following the procedure outlined in \cite[Section 5.1]{moskow} for nonzero $\lambda$, this problem is in turn equivalent to the nonlinear eigenvalue problem of finding $\lambda\in\mathbb{C}\setminus\{0\}$ and a nonzero $u\in H^1(B)$ such that
\begin{equation} \label{nldsop}
\lambda T_h(\lambda) u = u,
\end{equation}
where the operator $T_h(\lambda):H^1(B)\to H^1(B)$ is defined by
\begin{equation} \label{Th}
T_h(\lambda) = -\lambda^{-1}\mathbb{A}_h^{-1}\mathbb{K}_h + \mathbb{A}_h^{-1}\mathbb{B}.
\end{equation}
Writing the $\delta$-Stekloff eigenvalues in terms of a nonlinear eigenvalue problem allows us to apply the asymptotic formulas of \cite{moskow}, namely the following result (cf. \cite[Corollary 4.1]{moskow}), which we have adapted to the specific case of operators on a Hilbert space.

\begin{theorem} \label{theorem:moskow}

Let $\{T_h(\lambda)\}_{h\ge0}$ be a sequence of functions of $\lambda$ which are analytic in a region $U$ of the complex plane and have values in the space of compact linear operators on a Hilbert space $\mathcal{H}$ such that $T_h(\lambda)\to T_0(\lambda)$ in norm as $h\to0$, uniformly for $\lambda\in U$. Let $\lambda_0\in U$ be a nonzero simple nonlinear eigenvalue for $h=0$, define $DT_0(\lambda_0)$ as the derivative of $T_0$ with respect to $\lambda$ evaluated at $\lambda_0$, and let $u_0$ be the normalized eigenfunction corresponding to $\lambda_0$. For sufficiently small $h>0$, there exists a simple nonlinear eigenvalue $\lambda_h$ of $T_h$ such that, if
\begin{equation*}
\lambda_0^2 (DT_0(\lambda_0) u_0,u_0)_{\mathcal{H}} \neq -1,
\end{equation*}
then we have the formula
\begin{align*}
\lambda_h = \lambda_0 &+ \frac{\lambda_0^2 ((T_0(\lambda_0) - T_h(\lambda_0))u_0,u_0)_{\mathcal{H}}}{1 + \lambda_0^2(DT_0(\lambda_0)u_0,u_0)_{\mathcal{H}}} \\
&+ O\left( \sup_{\lambda\in U} \norm{(T_h(\lambda) - T_0(\lambda))|_{R(E_{\lambda_0})}} \norm{(T_h^*(\lambda) - T_0^*(\lambda))|_{R(E_{\lambda_0}^*)}} \right),
\end{align*}
where $E_{\lambda_0}$ is the spectral projection onto the one-dimensional eigenspace corresponding to $\lambda_0$.

\end{theorem}

We prove that $\{T_h(\lambda)\}_{h\ge0}$ as defined in \eqref{Th} satisfies the hypotheses of Theorem \ref{theorem:moskow} and compute the corresponding asymptotic formula. We begin with the following lemma, in which we state the basic properties of $\{T_h(\lambda)\}$.

\begin{lemma} \label{lemma:Th_basic}

For each $h\ge0$, the operator $T_h(\lambda)$ is compact for each $\lambda\in\mathbb{C}\setminus\{0\}$, and the mapping $\lambda\mapsto T_h(\lambda)$ is analytic in the domain $U = \mathbb{C}\setminus\overline{B_R(0)}$ for every $R>0$.

\end{lemma}

\begin{proof} We begin by showing that the operators $\mathbb{K}_h$ and $\mathbb{B}$ are both compact. For every $u\in H^1(B)$ we see from the assumption that $n_h\in L^\infty(B)$ and the Cauchy-Schwarz inequality that
\begin{align*}
\norm{\mathbb{K}_h u}_{H^1(B)} &= \sup_{u'\neq0} \frac{\abs{(\mathbb{K}_h u,u')_{H^1(B)}}}{\norm{u'}_{H^1(B)}} = k^2 \sup_{u'\neq0} \frac{\abs{((n_h+1) u,u')_B}}{\norm{u'}_{H^1(B)}} \le C_h \norm{u}_{L^2(B)},
\end{align*}
and by another application of the Cauchy-Schwarz inequality and the trace theorem we obtain
\begin{equation*}
\norm{\mathbb{B} u}_{H^1(B)} = \sup_{u'\neq0} \frac{\abs{(\mathbb{B} u,u')_{H^1(B)}}}{\norm{u'}_{H^1(B)}} = \sup_{u'\neq0} \frac{\abs{\inner{S_\delta u}{u'}_{\partial B}}}{\norm{u'}_{H^1(B)}} \le C\norm{S_\delta u}_{\partial B}.
\end{equation*}
Suppose that a sequence $\{u_j\}$ in $H^1(B)$ converges weakly to some $u_0\in H^1(B)$, in which case the compact embedding of $H^1(B)$ into $L^2(B)$ implies that $u_j\to u_0$ in $L^2(B)$ and the compact embedding of $H^{1/2}(\partial B)$ into $L^2(\partial B)$ and boundedness of $S_\delta$ imply that $S_\delta u_j\to S_\delta u_0$ in $L^2(\partial B)$. It follows that
\begin{equation*}
\norm{\mathbb{K}_h (u_j - u_0)}_{H^1(B)} \le C_h \norm{u_j - u_0}_{L^2(B)} \to 0
\end{equation*}
and
\begin{equation*}
\norm{\mathbb{B} (u_j - u_0)}_{H^1(B)} \le C \norm{S_\delta u_j - S_\delta u_0}_{\partial B} \to 0
\end{equation*}
as $j\to\infty$. As a consequence, we observe that $\mathbb{K}_h u_j\to \mathbb{K}_h u_0$ and $\mathbb{B}u_j\to\mathbb{B}u_0$ in $H^1(B)$, and we conclude that $\mathbb{K}_h$ and $\mathbb{B}$ are compact. Since composition with bounded operators preserves compactness, we have shown that $T_h(\lambda)$ is compact for each $\lambda\in\mathbb{C}\setminus\{0\}$. Analyticity of the mapping $\lambda\mapsto T_h(\lambda)$ for $\lambda\in U := \mathbb{C}\setminus\overline{B_R(0)}$ is clear from the definition of $T_h(\lambda)$. \qed

\end{proof}

Before we proceed to show norm convergence of $\{T_h(\lambda)\}$ to $T_0(\lambda)$, we require a regularity result that depends strongly on Assumption \ref{assumption:A}.

\begin{proposition} \label{prop:A}

If $w_h\in H^1(B)$ satisfies
\begin{subequations} \label{wh}
\begin{align}
\nabla\cdot A_h\nabla w_h - k^2 w_h &= f \text{ in } B, \\
\conormal{w_h}{A_h} &= \xi \text{ on } \partial B,
\end{align}
\end{subequations}
for given $f\in \tilde{H}^{s-1}(B)$ and $\xi\in H^{1/2}(\partial B)$, then there exists $\tau_h\in\left(0,\frac{1}{2}\right)$ such that for every $s\in[0,\tau_h)$ we have $w_h\in H^{s+1}(B)$ with the estimate
\begin{equation} \label{wh_est}
\norm{w_h}_{H^{s+1}(B)} \le C_{s,h} \left( \norm{f}_{\tilde{H}^{s-1}(B)} + \norm{\xi}_{H^{1/2}(\partial B)} \right).
\end{equation}
Moreover, if $A_h\to A_0$ in $W_\Sigma^{1,\infty}(B)$ as $h\to0$, then both $\tau_h$ and $C_{s,h}$ can be chosen independently of $h$ for sufficiently small $h$.

\end{proposition}

\begin{proof} We begin by defining a lifting function $\varphi\in H^1(B)$ as the unique solution of
\begin{align*}
\Delta\varphi - k^2\varphi &= 0 \text{ in } B, \\
\normal{\varphi} &= \xi \text{ on } \partial B,
\end{align*}
from which we observe that $w_h - \varphi\in H^1(B)$ satisfies
\begin{align*}
\nabla\cdot A_h\nabla(w_h-\varphi) - k^2 (w_h-\varphi) &= f + \nabla\cdot(I-A_h)\nabla\varphi \text{ in } B, \\
\conormal{(w_h-\varphi)}{A_h} &= 0 \text{ on } \partial B,
\end{align*}
with the boundary condition following from the assumption that $A_h = I$ in the open set $B\setminus\overline{D}$. Due to smoothness of $\partial B$, standard elliptic regularity estimates imply that $\varphi\in H^2(B)$ with $\norm{\varphi}_{H^2(B)} \le C\norm{\xi}_{H^{1/2}(\partial B)}$ (cf. \cite{adams}). It follows that $\nabla\varphi\in (H^1(B))^d$, which in turn implies that $\nabla\varphi\in (H^s(B))^d$ for all $s\in[0,1]$. By Proposition 2.1 in \cite{bonito_guermond_luddens}, from Assumption \ref{assumption:A} it follows that $A_h$ is a multiplier of the space $(H^s(B))^d$ for $s\in[0,\frac{1}{2})$, and we obtain $(I-A_h)\nabla\varphi\in(H^s(B))^d$ and hence $\nabla\cdot(I-A_h)\nabla\varphi\in \tilde{H}^{s-1}(B)$ for $s\in[0,\frac{1}{2})$. \par
Since $f+\nabla\cdot(I-A_h)\nabla\varphi\in \tilde{H}^{s-1}(B)$, we now apply Theorem 3.1 in \cite{bonito_guermond_luddens} to conclude that there exists $\tau_h\in\left(0,\frac{1}{2}\right)$ dependent upon $A_h$, $B$, and the partition $\{\Omega_m\}_{m=1}^M$ from Assumption \ref{assumption:A} such that $w_h-\varphi\in H^{s+1}(B)$ for $s\in[0,\tau_h)$ with the estimate
\begin{equation*}
\norm{w_h-\varphi}_{H^{s+1}(B)} \le C_{s,h} \norm{f + \nabla\cdot(I-A_h)\nabla\varphi}_{\tilde{H}^{s-1}(B)}.
\end{equation*}
The construction of $\varphi$ and its subsequent elliptic regularity estimate together imply \eqref{wh_est}. Finally, it can be seen from the proof of Theorem 3.1 in \cite{bonito_guermond_luddens} that convergence of $\{A_h\}$ to $A_0$ in $W_\Sigma^{1,\infty}(B)$ is sufficient to permit a choice of $\tau_h$ and $C_{s,h}$ that is independent of $h$ when $h$ is sufficiently small. \qed

\end{proof}

\begin{remark} \label{remark:A}

As a consequence of Proposition \ref{prop:A}, we will from this point forward assume that $A_h\to A_0$ in $W_\Sigma^{1,\infty}(B)$, allowing us to choose $\tau = \tau_h$ independently of $h$ and fix an element $s$ of $(0,\tau)$. 

\end{remark}

The following estimate concerning the operator $\mathbb{A}_h^{-1} - \mathbb{A}_0^{-1}$ will also be useful in many of the following results.

\begin{proposition} \label{prop:Adiff}

For sufficiently small $h>0$ we have the estimate
\begin{equation} \label{Adiff}
\abs{((\mathbb{A}_h^{-1} - \mathbb{A}_0^{-1})u,v)_{H^1(B)}} \le C_s \norm{A_h - A_0}_{(L^{d/s}(B))^{d\times d}} \norm{u}_{H^1(B)} \norm{v}_{H^{s+1}(B)}
\end{equation}
for all $u\in H^1(B)$ and $v\in H^{s+1}(B)$.

\end{proposition}

\begin{proof} For given $u\in H^1(B)$ and $v\in H^{s+1}(B)$, we define $v_h = \mathbb{A}_h^{-1} u$, and we observe that
\begin{equation*}
\abs{((\mathbb{A}_h^{-1} - \mathbb{A}_0^{-1})u,v)_{H^1(B)}} = \abs{(v_h - v_0,v)_{H^1(B)}} = \abs{((A_h - A_0)\nabla v_h,\nabla v)_B}.
\end{equation*}
Since $\nabla v\in (H^s(B))^d$, the Sobolev embedding theorem implies that $\nabla v\in (L^{2d/(d-2s)}(B))^d$ for $d=2,3$, in which case we may apply the three-term H{\"o}lder's inequality with $p = \frac{d}{s}$, $q = 2$, and $r = \frac{2d}{d-2s}$ to obtain
\begin{align*}
&\abs{((A_h - A_0)\nabla v_h,\nabla v)_B} \\
&\hspace{4em} \le \norm{A_h - A_0}_{(L^{d/s}(B))^{d\times d}} \norm{\nabla v_h}_{(L^2(B))^d} \norm{\nabla v}_{(L^{2d/(d-2s)}(B))^d} \\
&\hspace{4em} \le C_s \norm{A_h - A_0}_{(L^{d/s}(B))^{d\times d}} \norm{v_h}_{H^1(B)} \norm{v}_{H^{s+1}(B)}.
\end{align*}
Finally, we have
\begin{align*}
\norm{v_h}_{H^1(B)}^2 &= (A_0\nabla v_h,\nabla v_h)_B + k^2 (v_h,v_h)_B \\
&= -((A_h - A_0)\nabla v_h,\nabla v_h)_B + (u,v_h)_{H^1(B)} \\
&\le \norm{A_h - A_0}_{(L^\infty(B))^{d\times d}} \norm{v_h}_{H^1(B)}^2 + \norm{u}_{H^1(B)} \norm{v_h}_{H^1(B)},
\end{align*}
where we applied the three-term H{\"o}lder's inequality with $p=\infty$ and $q=r=2$ to obtain the last inequality. Convergence of $\{A_h\}$ to $A_0$ in $(L^\infty(B))^{d\times d}$ implies that $\norm{v_h}_{H^1(B)} \le C \norm{u}_{H^1(B)}$ for sufficiently small $h$, from which we obtain \eqref{Adiff}. \qed

\end{proof}

\begin{remark} \label{remark:Adiff}

Since $\mathbb{A}_h^{-1} - \mathbb{A}_0^{-1}$ is a self-adjoint operator, the estimate \eqref{Adiff} also holds with the roles of $u$ and $v$ interchanged on the right-hand side, provided that each lies in the appropriate space.

\end{remark}

We now estimate the norm $\norm{T_h(\lambda) - T_0(\lambda)}$ for $\lambda\in U$.

\begin{lemma} \label{lemma:Th_norm}

If $A_h\to A_0$ in $W_\Sigma^{1,\infty}(B)$ and $n_h\to n_0$ in $L^2(B)$ as $h\to0$, then the sequence $\{T_h(\lambda)\}$ satisfies the norm estimate
\begin{equation} \label{Th_norm}
\norm{T_h(\lambda) - T_0(\lambda)} \le C_s \left( \norm{A_h - A_0}_{(L^{d/s}(B))^{d\times d}} + \norm{n_h - n_0}_{L^{p_0}(B)} \right),
\end{equation}
where for each $\epsilon>0$ we have
\begin{equation} \label{p0}
p_0 = \left\{\begin{array}{cl} 1+\epsilon &\text{ if } d=2, \\ \frac{3}{2} &\text{ if } d=3, \end{array} \right.
\end{equation}
and the constant $C_s$ is independent of $h$ and $\lambda\in U = \mathbb{C}\setminus\overline{B_R(0)}$ but depends on $\epsilon$ whenever $d=2$. As a result, we have $T_h(\lambda)\to T_0(\lambda)$ in norm as $h\to0$, uniformly for $\lambda\in U$.

\end{lemma}

\begin{proof} Let $\lambda\in U$. We first observe that we may write
\begin{equation*}
T_h(\lambda) - T_0(\lambda) = -(\mathbb{A}_h^{-1} - \mathbb{A}_0^{-1})(\lambda^{-1}\mathbb{K}_h - \mathbb{B}) - \lambda^{-1}\mathbb{A}_0^{-1}(\mathbb{K}_h - \mathbb{K}_0),
\end{equation*}
immediately providing the estimate
\begin{align*}
\abs{((T_h(\lambda) - T_0(\lambda))u,u')_{H^1(B)}} &\le \abs{((\mathbb{A}_h^{-1} - \mathbb{A}_0^{-1})(\lambda^{-1}\mathbb{K}_h - \mathbb{B})u,u')_{H^1(B)}} \\
&\hspace{4em} + \abs{\lambda}^{-1} \abs{(\mathbb{A}_0^{-1}(\mathbb{K}_h - \mathbb{K}_0)u,u')_{H^1(B)}}
\end{align*}
for all $u,u'\in H^1(B)$. In order to estimate the first term, we notice that $(\lambda^{-1}\mathbb{K}_h - \mathbb{B})u$ satisfies \eqref{wh} with $f = \lambda^{-1} k^2(n_h+1)u\in L^2(B)$ and $\xi = S_\delta u\in H^{1/2}(\partial B)$, from which Proposition \ref{prop:A} implies that $(\lambda^{-1}\mathbb{K}_h - \mathbb{B})u\in H^{s+1}(B)$ with the estimate
\begin{align*}
\norm{(\lambda^{-1}\mathbb{K}_h - \mathbb{B})u}_{H^{s+1}(B)} &\le C_s \left( \norm{\lambda^{-1} k^2(n_h+1)u}_{\tilde{H}^{s-1}(B)} + \norm{S_\delta u}_{H^{1/2}(\partial B)} \right) \\
&\le C_s \left( \norm{n_h u}_{\tilde{H}^{s-1}(B)} + \norm{u}_{H^1(B)} \right)
\end{align*}
with $C_s$ independent of $\lambda\in U$. We apply a duality argument and the three-term H{\"o}lder's inequality to estimate $\norm{n_h u}_{\tilde{H}^{s-1}(B)}$. By definition of the dual norm $\norm{\cdot}_{\tilde{H}^{s-1}(B)}$ we have
\begin{equation*}
\norm{n_h u}_{\tilde{H}^{s-1}(B)} = \sup_{\norm{\psi}_{H^{1-s}(B)} = 1} \abs{ (n_h u,\psi)_B },
\end{equation*}
where we interpret the duality pairing in terms of the Gelfand triple $H^{1-s}(B) \subseteq L^2(B) \subseteq \tilde{H}^{s-1}(B)$. Since $u\in H^1(B)$ and $\psi\in H^{1-s}(B)$, the Sobolev embedding theorem implies that $u\in L^q(B)$ for all $q\in[2,\infty)$ and $\psi\in L^{2/s}(B)$ when $d=2$ and that $u\in L^6(B)$ and $\psi\in L^{6/(1+2s)}(B)$ when $d=3$, each with a continuous embedding. For $d=2$, applying the three-term H{\"o}lder's inequality with $p = \frac{3}{2-s}$, $q = \frac{6}{2-s}$, and $r = \frac{2}{s}$ yields
\begin{align*}
\abs{ (n_h u,\psi)_{\tilde{H}^{s-1}(B)} } &\le \norm{n_h}_{L^{3/(2-s)}(B)} \norm{u}_{L^{6/(2-s)}(B)} \norm{\psi}_{L^{2/s}(B)} \\
&\le C_s \norm{n_h}_{L^{3/(2-s)}(B)} \norm{u}_{H^1(B)} \norm{\psi}_{H^{1-s}(B)}.
\end{align*}
For $d=3$, we apply the three-term H{\"o}lder's inequality with $p = \frac{3}{2-s}$, $q = 6$, and $r = \frac{6}{1+2s}$ to obtain
\begin{align*}
\abs{ (n_h u,\psi)_{\tilde{H}^{s-1}(B)} } &\le \norm{n_h}_{L^{3/(2-s)}(B)} \norm{u}_{L^6(B)} \norm{\psi}_{L^{6/(1+2s)}(B)} \\
&\le C_s \norm{n_h}_{L^{3/(2-s)}(B)} \norm{u}_{H^1(B)} \norm{\psi}_{H^{1-s}(B)}.
\end{align*}
Thus, in either case we have $\norm{n_h u}_{\tilde{H}^{s-1}(B)} \le C_s \norm{n_h}_{L^{3/(2-s)}(B)} \norm{u}_{H^1(B)}$. Convergence of $\{n_h\}$ to $n_0$ in $L^2(B)$ implies uniform boundedness of $\{n_h\}$ in the weaker $L^{3/(2-s)}(B)$-norm, and hence there exists a constant $C_s$ independent of $h$ for which
\begin{equation*}
\norm{(\lambda^{-1}\mathbb{K}_h - \mathbb{B})u}_{H^{s+1}(B)} \le C_s \norm{u}_{H^1(B)}.
\end{equation*}
We now proceed to the main estimate. Applying the estimate from Proposition \ref{prop:Adiff} (with the functions reversed as in Remark \ref{remark:Adiff}) yields
\begin{align*}
&\abs{((\mathbb{A}_h^{-1} - \mathbb{A}_0^{-1})(\lambda^{-1}\mathbb{K}_h - \mathbb{B})u,u')_{H^1(B)}} \\
&\hspace{4em}\le C_s \norm{A_h - A_0}_{(L^{d/s}(B))^{d\times d}} \norm{(\lambda^{-1}\mathbb{K}_h - \mathbb{B})u}_{H^{s+1}(B)} \norm{u'}_{H^1(B)} \\
&\hspace{4em} \le C_s \norm{A_h - A_0}_{(L^{d/s}(B))^{d\times d}} \norm{u}_{H^1(B)} \norm{u'}_{H^1(B)}.
\end{align*}
For the second term we recall that $\mathbb{A}_0$ is equal to the identity operator, which allows us to write
\begin{equation*}
\abs{(\mathbb{A}_0^{-1}(\mathbb{K}_h - \mathbb{K}_0)u,u')_{H^1(B)}} = k^2\abs{((n_h - n_0)u,u')}_B.
\end{equation*}
Since $u,u'\in H^1(B)$, the Sobolev embedding theorem implies that $u,u'\in L^q(B)$, where $q\in[2,\infty)$ for $d=2$ and $q = 6$ for $d=3$. In the case $d=2$, we apply the three-term H{\"o}lder's inequality for a given $\epsilon>0$ with $p = 1+\epsilon$ and $q = r = \frac{2(1+\epsilon)}{\epsilon}$, and in the case $d=3$ we apply the inequality with $p = \frac{3}{2}$ and $q=r=6$ to obtain
\begin{align*}
\abs{((n_h-n_0)u,u')_B} &\le \norm{n_h - n_0}_{L^{p_0}(B)} \norm{u}_{L^q(B)} \norm{u'}_{L^r(B)} \\
&\le C \norm{n_h - n_0}_{L^{p_0}(B)} \norm{u}_{H^1(B)} \norm{u'}_{H^1(B)},
\end{align*}
where $p_0$ is given by \eqref{p0}. Combining these estimates and noting that all constants are independent of $h$ and $\lambda\in U$, the estimate \eqref{Th_norm} follows from the definition of the operator norm. Moreover, convergence of $\{A_h\}$ to $A_0$ and of $\{n_h\}$ to $n_0$ in the spaces $W_\Sigma^{1,\infty}(B)$ and $L^2(B)$, respectively, implies that the right-hand side of \eqref{Th_norm} converges to $0$ as $h\to0$, proving the final assertion. \qed

\end{proof}

With the results of Lemmas \ref{lemma:Th_basic} and \ref{lemma:Th_norm} in hand, we proceed to apply Theorem \ref{theorem:moskow} to the sequence $\{T_h(\lambda)\}$.

\begin{theorem} \label{theorem:main}

Suppose that $A_h\to A_0$ in $W_\Sigma^{1,\infty}(B)$ and that $n_h\to n_0$ in $L^2(B)$ as $h\to0$, and suppose that there exist no nontrivial solutions of \eqref{dir} for $(A,n) = (A_0,n_0)$. Let $\lambda_0$ be a nonzero simple $\delta$-Stekloff eigenvalue for $(A,n) = (A_0,n_0)$ with $H^1(B)$-normalized eigenfunction $u_0$, choose $R>0$ such that $R<\abs{\lambda_0}$, and let $U = \mathbb{C}\setminus\overline{B_R(0)}$. For sufficiently small $h>0$ there exists a simple $\delta$-Stekloff eigenvalue $\lambda_h$ for $(A,n) = (A_h,n_h)$ that satisfies the formula
\begin{align}
\begin{split} \label{Th_main}
\lambda_h = \lambda_0 &+ \frac{-((A_h-A_0)\nabla u_0,\nabla u_0)_B + k^2 ((n_h-n_0)u_0,u_0)_B}{\inner{S_\delta u_0}{u_0}_{\partial B}} \\
&\hspace{3em} + O\left( \norm{A_h-A_0}_{(L^{d/s}(B))^{d\times d}}^2 + \norm{n_h-n_0}_{L^{p'}(B)}^2 \right),
\end{split}
\end{align}
where the exponent $p'$ is given by
\begin{equation} \label{pprime}
p' = \left\{ \begin{array}{cl} 1+\epsilon &\text{ if } d=2, \\ \frac{3}{2+s} &\text{ if } d=3, \end{array} \right.
\end{equation}
for a given $\epsilon>0$.

\end{theorem}

\begin{proof} We begin with the observation that Lemmas \ref{lemma:Th_basic} and \ref{lemma:Th_norm} verify that the hypotheses of Theorem \ref{theorem:moskow} are satisfied by the sequence $\{T_h(\lambda)\}$ for $U = \mathbb{C}\setminus\overline{B_R(0)}$ with $R<\abs{\lambda_0}$. Moreover, from the assumption that $u_0$ is an $H^1(B)$-normalized eigenfunction corresponding to $\lambda_0$, we obtain
\begin{align*}
1+\lambda_0^2 (DT_0(\lambda_0) u_0,u_0)_{H^1(B)} &= 1 + \lambda_0^2(\lambda_0^{-2}\mathbb{A}_0^{-1}\mathbb{K}_0 u_0,u_0)_{H^1(B)} \\
&= 1 - (\lambda_0 T_0(\lambda_0) u_0,u_0)_{H^1(B)} + \lambda_0(\mathbb{B} u_0,u_0)_{H^1(B)} \\
&= 1 - (u_0,u_0)_{H^1(B)} - \lambda_0\inner{S_\delta u_0}{u_0}_{\partial B} \\
&= - \lambda_0\inner{S_\delta u_0}{u_0}_{\partial B}.
\end{align*}
Since $\lambda_0\neq0$ and $u_0\neq0$ by assumption (for otherwise $u_0$ would be a nontrivial solution of \eqref{dir} for $(A,n) = (A_0,n_0)$), it follows that $\lambda_0\inner{S_\delta u_0}{u_0}_{\partial B} \neq 0$, and by Theorem \ref{theorem:moskow} there exists a simple $\delta$-Stekloff eigenvalue $\lambda_h$ for $(A,n) = (A_h,n_h)$ such that
\begin{align*}
\lambda_h = \lambda_0 &+ \frac{\lambda_0^2 ((T_0(\lambda_0) - T_h(\lambda_0))u_0,u_0)_{H^1(B)}}{1 + \lambda_0^2(DT_0(\lambda_0)u_0,u_0)_{H^1(B)}} \\
&+ O\left( \sup_{\lambda\in U} \norm{(T_h(\lambda) - T_0(\lambda))|_{R(E_{\lambda_0})}} \norm{(T_h^*(\lambda) - T_0^*(\lambda))|_{R(E_{\lambda_0}^*)}} \right).
\end{align*}
We first estimate the remainder term
\begin{equation*}
\mathcal{R}_h = \sup_{\lambda\in U} \norm{(T_h(\lambda) - T_0(\lambda))|_{R(E_{\lambda_0})}} \norm{(T_h^*(\lambda) - T_0^*(\lambda))|_{R(E_{\lambda_0}^*)}}.
\end{equation*}
For the leftmost norm, we may begin as in the proof of Lemma \ref{lemma:Th_norm} with the estimate
\begin{align*}
\abs{((T_h(\lambda) - T_0(\lambda))u,u')_{H^1(B)}} &\le \abs{((\mathbb{A}_h^{-1} - \mathbb{A}_0^{-1})(\lambda^{-1}\mathbb{K}_h - \mathbb{B})u,u')_{H^1(B)}} \\
&\hspace{4em} + \abs{\lambda}^{-1} \abs{(\mathbb{A}_0^{-1}(\mathbb{K}_h - \mathbb{K}_0)u,u')_{H^1(B)}}
\end{align*}
for $u\in R(E_{\lambda_0})$ and $u'\in H^1(B)$. The first term may be estimated as in that proof to obtain
\begin{align*}
&\abs{((\mathbb{A}_h^{-1} - \mathbb{A}_0^{-1})(\lambda^{-1}\mathbb{K}_h - \mathbb{B})u,u')_{H^1(B)}} \\
&\hspace{4em} \le C_s \norm{A_h - A_0}_{(L^{d/s}(B))^{d\times d}} \norm{u}_{H^1(B)} \norm{u'}_{H^1(B)}.
\end{align*}
We may proceed as before to estimate the second term, but Proposition \ref{prop:A} now implies that $u\in H^{s+1}(B)$ with the estimate
\begin{equation} \label{u_est}
\norm{u}_{H^{s+1}(B)} \le C_s \norm{u}_{H^1(B)},
\end{equation}
and we may take advantage of this higher regularity. For $d=2$ the Sobolev embedding theorem implies that $u\in C_b(B)$ and $u'\in L^r(B)$ for every $r\in[2,\infty)$, and applying the three-term H{\"o}lder's inequality with $p = \frac{r}{r-1}$, $q=\infty$, and any $r\in[2,\infty)$ yields
\begin{align*}
&\abs{\lambda^{-1}(\mathbb{A}_0^{-1}(\mathbb{K}_h - \mathbb{K}_0)u,u')_{H^1(B)}} \\
&\hspace{6em} = k^2\abs{\lambda}^{-1} \abs{((n_h-n_0)u,u')_B} \\
&\hspace{6em} \le k^2 R^{-1} \norm{n_h-n_0}_{L^{r/(r-1)}(B)} \norm{u}_{C_b(B)} \norm{u'}_{L^r(B)} \\
&\hspace{6em} \le C_{s,\epsilon} \norm{n_h-n_0}_{L^{1+\epsilon}(B)} \norm{u}_{H^1(B)} \norm{u'}_{H^1(B)},
\end{align*}
where in the last inequality we chose $r = \frac{1+\epsilon}{\epsilon}$ for a given $\epsilon>0$ and applied the Sobolev embedding estimates along with \eqref{u_est}. For $d=3$ we have $u\in L^{6/(1-2s)}(B)$ and $u'\in L^6(B)$, and by applying the three-term H{\"o}lder's inequality with $p = \frac{3}{2+s}$, $q = \frac{6}{1-2s}$, and $r = 6$ we obtain
\begin{align*}
&\abs{\lambda^{-1}(\mathbb{A}_0^{-1}(\mathbb{K}_h - \mathbb{K}_0)u,u')_{H^1(B)}} \\
&\hspace{6em} = k^2\abs{\lambda}^{-1} \abs{((n_h-n_0)u,u')_B} \\
&\hspace{6em} \le k^2 R^{-1} \norm{n_h-n_0}_{L^{3/(2+s)}(B)} \norm{u}_{L^{6/(1-2s)}(B)} \norm{u'}_{L^6(B)} \\
&\hspace{6em} \le C_{s,\epsilon} \norm{n_h-n_0}_{L^{3/(2+s)}(B)} \norm{u}_{H^1(B)} \norm{u'}_{H^1(B)},
\end{align*}
By combining these estimates along with the definition of the operator norm, we arrive at the estimate
\begin{align}
\begin{split} \label{Th1}
&\norm{(T_h(\lambda) - T_0(\lambda))|_{R(E_{\lambda_0})}} \\
&\hspace{2em} \le C_s\biggr( \norm{A_h - A_0}_{(L^{d/s}(B))^{d\times d}} + \norm{n_h-n_0}_{L^{p'}(B)} \biggr),
\end{split}
\end{align}
where $p'$ is defined as in \eqref{pprime}. The constant $C_s$ is uniform with respect to $h$ and $\lambda$, but it depends on $\epsilon$ in the case $d=2$. We now estimate the rightmost norm in the remainder term $\mathcal{R}_h$, which is in terms of the sequence $\{T_h^*(\lambda)\}$ of adjoint operators. Since the operators $\mathbb{A}_h$ and $\mathbb{B}$ are self-adjoint, we may write the operator $T_h^*(\lambda) - T_0^*(\lambda)$ in the form
\begin{equation*}
T_h^*(\lambda) - T_0^*(\lambda) = -(\overline{\lambda}^{-1} \mathbb{K}_h^* - \mathbb{B}) (\mathbb{A}_h^{-1} - \mathbb{A}_0^{-1}) - \overline{\lambda}^{-1} (\mathbb{K}_h^* - \mathbb{K}_0^*)\mathbb{A}_0^{-1}.
\end{equation*}
For the first term, we observe that
\begin{align*}
&\abs{((\overline{\lambda}^{-1} \mathbb{K}_h^* - \mathbb{B}) (\mathbb{A}_h^{-1} - \mathbb{A}_0^{-1})u^*,u')_{H^1(B)}} \\
&\hspace{8em} = \abs{((\mathbb{A}_h^{-1} - \mathbb{A}_0^{-1})u^*,(\lambda^{-1} \mathbb{K}_h - \mathbb{B})u')_{H^1(B)}}
\end{align*}
for $u^*\in R(E_{\lambda_0})$ and $u'\in H^1(B)$. Recalling from the proof Lemma \ref{lemma:Th_norm} that $(\lambda^{-1} \mathbb{K}_h - \mathbb{B})$ maps $H^1(B)$ continuously into $H^{s+1}(B)$, uniformly for sufficiently small $h$, we apply Proposition \ref{prop:Adiff} to obtain
\begin{align*}
&\abs{((\overline{\lambda}^{-1} \mathbb{K}_h^* - \mathbb{B}) (\mathbb{A}_h^{-1} - \mathbb{A}_0^{-1})u^*,u')_{H^1(B)}} \\
&\hspace{6em} \le C_s \norm{A_h - A_0}_{(L^{d/s}(B))^{d\times d}} \norm{u^*}_{H^1(B)} \norm{u'}_{H^1(B)}.
\end{align*}
For the second term, we note that $u^*\in R(E_{\lambda_0}^*)$ satisfies \eqref{deltastek} with $(A,n) = (A_0,\overline{n_0})$ and $\lambda = \overline{\lambda_0}$, from which Proposition \ref{prop:A} implies that $u^*\in H^{s+1}(B)$ with the estimate $\norm{u^*}_{H^{s+1}(B)} \le C_s \norm{u^*}_{H^1(B)}$. Observing that
\begin{equation*}
((\mathbb{K}_h^* - \mathbb{K}_0^*)\mathbb{A}_0^{-1}u^*,u')_{H^1(B)} = -k^2((\overline{n_h} - \overline{n_0})u^*,u')_{H^1(B)},
\end{equation*}
we may apply the three-term H{\"o}lder's inequality to the second term just as we did when estimating the leftmost term in $\mathcal{R}_h$ to obtain
\begin{equation*}
\abs{(\overline{\lambda}^{-1} (\mathbb{K}_h^* - \mathbb{K}_0^*)\mathbb{A}_0^{-1}u^*,u')_{H^1(B)}} \le C_s \norm{n_h - n_0}_{L^{p'}(B)} \norm{u^*}_{H^1(B)} \norm{u'}_{H^1(B)}.
\end{equation*}
By combining these estimates, the definition of the operator norm yields
\begin{align}
\begin{split} \label{Th2}
&\norm{(T_h^*(\lambda) - T_0^*(\lambda))|_{R(E_{\lambda_0}^*)}} \\
&\hspace{2em} \le C_s\biggr( \norm{A_h - A_0}_{(L^{d/s}(B))^{d\times d}} + \norm{n_h-n_0}_{L^{p'}(B)} \biggr),
\end{split}
\end{align}
where $p'$ is defined as in \eqref{pprime}, and the constant $C_s$ is uniform with respect to $h$ and $\lambda$ but depends on $\epsilon$ in the case $d=2$. Noticing that the right-hand sides of \eqref{Th1} and \eqref{Th2} coincide, we arrive at the result
\begin{equation} \label{remainder}
\mathcal{R}_h = O\left( \norm{A_h-A_0}_{(L^{d/s}(B))^{d\times d}}^2 + \norm{n_h-n_0}_{L^{p'}(B)}^2 \right).
\end{equation}
In order to evaluate the correction term, we write the numerator in the form
\begin{subequations} \label{Tdiff}
\begin{align}
&((T_0(\lambda_0) - T_h(\lambda_0))u_0,u_0)_{H^1(B)} \nonumber\\
&\hspace{4em} = -((\mathbb{A}_h - \mathbb{A}_0)(\lambda_0^{-1}\mathbb{K}_0 - \mathbb{B})u_0,u_0)_{H^1(B)} \label{Tdiff1} \\
&\hspace{6em} + \lambda_0^{-1}((\mathbb{K}_h - \mathbb{K}_0)u_0,u_0)_{H^1(B)} \label{Tdiff2} \\
&\hspace{6em} - ((\mathbb{A}_h^{-1} - \mathbb{A}_0^{-1})(\mathbb{A}_h - \mathbb{A}_0)(\lambda_0^{-1}\mathbb{K}_0 - \mathbb{B})u_0,u_0)_{H^1(B)} \label{Tdiff3} \\
&\hspace{6em} + \lambda_0^{-1}((\mathbb{A}_h^{-1} - \mathbb{A}_0^{-1})(\mathbb{K}_h - \mathbb{K}_0)u_0,u_0)_{H^1(B)}, \label{Tdiff4}
\end{align}
\end{subequations}
where we have used the identity
\begin{equation*}
\mathbb{A}_h^{-1} - \mathbb{A}_0^{-1} = -\mathbb{A}_0^{-1}(\mathbb{A}_h - \mathbb{A}_0)\mathbb{A}_0^{-1} - (\mathbb{A}_h^{-1} - \mathbb{A}_0^{-1})(\mathbb{A}_h - \mathbb{A}_0)\mathbb{A}_0^{-1}
\end{equation*}
to force all terms involving $\mathbb{A}_h^{-1} - \mathbb{A}_0^{-1}$ to be sufficiently small that they may be absorbed into the remainder term, as we shall show. We note that we also used the fact that $\mathbb{A}_0^{-1}$ is the identity operator to ignore it anytime it is composed with another operator; however, we have left it in place when appearing in a difference of operators in order to clearly see the influence of the perturbation. We first use Proposition \ref{prop:Adiff} to show that \eqref{Tdiff3} and \eqref{Tdiff4} may be absorbed into the remainder term $\mathcal{R}_h$. Indeed, since $u_0$ is an eigenfunction corresponding to $\lambda_0$, we have $(\lambda_0^{-1} \mathbb{K}_0 - \mathbb{B})u_0 = -\lambda_0^{-1} u_0$ and $u_0\in H^{s+1}(B)$ by Proposition \ref{prop:A}. Thus, Proposition \ref{prop:Adiff} asserts that
\begin{align*}
&\abs{((\mathbb{A}_h^{-1} - \mathbb{A}_0^{-1})(\mathbb{A}_h - \mathbb{A}_0)(\lambda_0^{-1}\mathbb{K}_0 - \mathbb{B})u_0,u_0)_{H^1(B)}} \\
&\hspace{6em} = \abs{\lambda_0}^{-1} \abs{((\mathbb{A}_h^{-1} - \mathbb{A}_0^{-1})(\mathbb{A}_h - \mathbb{A}_0)u_0,u_0)_{H^1(B)}} \\
&\hspace{6em} \le C_s \norm{A_h - A_0}_{(L^{d/s}(B))^{d\times d}} \norm{(\mathbb{A}_h - \mathbb{A}_0)u_0}_{H^1(B)} \norm{u_0}_{H^{s+1}(B)}.
\end{align*}
We observe that $\norm{u_0}_{H^{s+1}(B)} \le C_s \norm{u_0}_{H^1(B)} = C_s$ by Proposition \ref{prop:A} and that
\begin{align*}
\norm{(\mathbb{A}_h - \mathbb{A}_0)u_0}_{H^1(B)} &= \sup_{\norm{\psi}_{H^1(B)}=1} \abs{((A_h - A_0)\nabla u_0,\nabla\psi)_B} \\
&\le C_s \norm{A_h - A_0}_{(L^{d/s}(B))^{d\times d}},
\end{align*}
where we have applied the three-term H{\"o}lder's inequality with $p = \frac{d}{s}$, $q = \frac{2d}{d-2s}$, and $r = 2$ in a similar manner to the proof of Proposition \ref{prop:Adiff}. It follows that \eqref{Tdiff3} satisfies the estimate
\begin{equation*}
\abs{((\mathbb{A}_h^{-1} - \mathbb{A}_0^{-1})(\mathbb{A}_h - \mathbb{A}_0)(\lambda_0^{-1}\mathbb{K}_0 - \mathbb{B})u_0,u_0)_{H^1(B)}} \le C_s \norm{A_h - A_0}_{(L^{d/s}(B))^{d\times d}}^2
\end{equation*}
and may be absorbed into the remainder term $\mathcal{R}_h$. The estimate of \eqref{Tdiff4} follows similarly, as from Proposition \ref{prop:Adiff} we have
\begin{align*}
&\abs{\lambda_0^{-1}((\mathbb{A}_h^{-1} - \mathbb{A}_0^{-1})(\mathbb{K}_h - \mathbb{K}_0)u_0,u_0)_{H^1(B)}} \\
&\hspace{4em} \le C_s \norm{A_h - A_0}_{(L^{d/s}(B))^{d\times d}} \norm{(\mathbb{K}_h - \mathbb{K}_0)u_0}_{H^1(B)} \norm{u_0}_{H^{s+1}(B)}.
\end{align*}
As with the previous estimate, the bound $\norm{u_0}_{H^{s+1}(B)} \le C_s$ implies that
\begin{align*}
\norm{(\mathbb{K}_h - \mathbb{K}_0)u_0}_{H^1(B)} &= \sup_{\norm{\psi}_{H^1(B)}=1} \abs{((n_h-n_0)u_0,\psi)_{H^1(B)}} \\
&\le C_s \norm{n_h - n_0}_{L^{p'}(B)},
\end{align*}
where we have applied the three-term H{\"o}lder's inequality as we did in the estimate of the remainder $\mathcal{R}_h$. From the arithmetic-geometric mean inequality we obtain
\begin{align*}
&\abs{\lambda_0^{-1}((\mathbb{A}_h^{-1} - \mathbb{A}_0^{-1})(\mathbb{K}_h - \mathbb{K}_0)u_0,u_0)_{H^1(B)}} \\
&\hspace{6em} \le C_s \norm{A_h - A_0}_{(L^{d/s}(B))^{d\times d}} \norm{n_h - n_0}_{L^{p'}(B)} \\
&\hspace{6em} \le C_s \biggr( \norm{A_h - A_0}_{(L^{d/s}(B))^{d\times d}}^2 + \norm{n_h - n_0}_{L^{p'}(B)}^2 \biggr),
\end{align*}
and we conclude that \eqref{Tdiff4} may also be absorbed into the remainder term $\mathcal{R}_h$. Finally, we compute the correction term from the remaining terms \eqref{Tdiff1} and \eqref{Tdiff2}. Again noting that $(\lambda_0^{-1}\mathbb{K}_0 - \mathbb{B})u_0 = -\lambda_0^{-1} u_0$, we may write \eqref{Tdiff1} as
\begin{align*}
-((\mathbb{A}_h - \mathbb{A}_0)(\lambda_0^{-1}\mathbb{K}_0 - \mathbb{B})u_0,u_0)_{H^1(B)} &= \lambda_0^{-1} ((\mathbb{A}_h - \mathbb{A}_0)u_0,u_0)_{H^1(B)} \\
&= \lambda_0^{-1} ((A_h - A_0)\nabla u_0,\nabla u_0)_B,
\end{align*}
and we may write \eqref{Tdiff2} in the form
\begin{equation*}
\lambda_0^{-1} ((\mathbb{K}_h - \mathbb{K}_0)u_0,u_0)_{H^1(B)} = -\lambda_0^{-1} k^2((n_h - n_0)u_0,u_0)_B.
\end{equation*}
Thus, we arrive at the correction formula given by \eqref{Th_main}. \qed

\end{proof}

In the following corollary, we use \eqref{Th_main} to obtain an upper bound on perturbations of the eigenvalues.

\begin{corollary} \label{corollary:eig_est}

Assume the hypotheses of Theorem \ref{theorem:main}, and define the exponent $p''$ by
\begin{equation} \label{pdoubleprime}
p'' = \left\{ \begin{array}{cl} 1 &\text{ if } d=2, \\ \frac{3}{2(s+1)} &\text{ if } d=3. \end{array} \right.
\end{equation}
For sufficiently small $h>0$ we have the bound
\begin{equation} \label{eig_est}
\abs{\lambda_h - \lambda_0} \le C_s \biggr( \norm{A_h-A_0}_{(L^{d/2s}(B))^{d\times d}} + \norm{n_h-n_0}_{L^{p''}(B)} \biggr).
\end{equation}

\end{corollary}

\begin{proof} We first direct our attention to estimating the correction term in \eqref{Th_main}. Since $\nabla u_0\in H^{s+1}(B)$ with the estimate $\norm{\nabla u_0}_{(H^s(B))^d} \le C_s$, implying by the Sobolev embedding theorem that $\nabla u_0\in (L^{2d/(d-2s)}(B))^d$, we may apply the three-term H{\"o}lder's inequality with $p = \frac{d}{2s}$ and $q = r = \frac{2d}{d-2s}$ to obtain
\begin{equation*}
\abs{(A_h-A_0)\nabla u_0,\nabla u_0)_B} \le C_s \norm{A_h - A_0}_{(L^{d/2s}(B))^{d\times d}}.
\end{equation*}
For the second term, the Sobolev embedding theorem implies that $u_0\in C_b(B)$ if $d=2$ and $u_0\in L^{6/(1-2s)}(B)$ if $d=3$. For $d=2$, we use the observation that $u_0$ is uniformly bounded in $B$ to obtain
\begin{equation*}
\abs{((n_h-n_0)u_0,u_0)_B} \le \norm{n_h-n_0}_{L^1(B)} \norm{u_0}_{C_b(B)}^2 \le C_s \norm{n_h-n_0}_{L^1(B)}.
\end{equation*}
For $d=3$, the three-term H{\"o}lder's inequality with $p = \frac{3}{2(s+1)}$ and $q = r = \frac{6}{1-2s}$ yields
\begin{align*}
\abs{((n_h-n_0)u_0,u_0)_B} &\le \norm{n_h - n_0}_{L^{3/2(s+1)}(B)} \norm{u_0}_{L^{6/(1-2s)}(B)}^2 \\
&\le C_s \norm{n_h - n_0}_{L^{3/2(s+1)}(B)}.
\end{align*}
Finally, we observe that the remainder term may be absorbed into the upper bound for the correction term provided that
\begin{subequations} \label{rem}
\begin{align}
\norm{A_h-A_0}_{(L^{d/s}(B))^{d\times d}}^2 &\le C\norm{A_h-A_0}_{(L^{d/2s}(B))^{d\times d}}, \label{rem1} \\
\norm{n_h-n_0}_{L^{p'}(B)}^2 &\le C\norm{n_h-n_0}_{L^{p''}(B)}, \label{rem2}
\end{align}
\end{subequations}
for sufficiently small $h$. We use H{\"o}lder's interpolation inequality (Theorem \ref{theorem:holder_interp}) to show that both \eqref{rem1} and \eqref{rem2} hold. Beginning with \eqref{rem1}, we first note that
\begin{equation*}
\norm{A_h - A_0}_{(L^{d/s}(B))^{d\times d}}^2 = \left( \sum_{i,j=1}^d \norm{(a_h)_{ij} - (a_0)_{ij}}_{L^{d/s}(B)} \right)^2,
\end{equation*}
where $(a_h)_{ij}$ represents the $i,j$ entry of the matrix $A_h$. We apply H{\"o}lder's interpolation inequality with $p = \frac{d}{s}$, $q = \infty$, $r = \frac{d}{2s}$, and $\alpha = \frac{1}{2}$ to each term in the sum to obtain
\begin{align*}
&\norm{A_h - A_0}_{(L^{d/s}(B))^{d\times d}}^2 \\
&\hspace{4em} \le \left( \sum_{i,j=1}^d \norm{(a_h)_{ij} - (a_0)_{ij}}_{L^{d/2s}(B)}^{1/2} \norm{(a_h)_{ij} - (a_0)_{ij}}_{L^\infty(B)}^{1/2} \right)^2,
\end{align*}
and an application of the Cauchy-Schwarz inequality yields
\begin{align*}
&\norm{A_h - A_0}_{(L^{d/s}(B))^{d\times d}}^2 \\
&\hspace{4em} \le \left( \sum_{i,j=1}^d \norm{(a_h)_{ij} - (a_0)_{ij}}_{L^{d/2s}(B)} \right) \left( \sum_{i,j=1}^d \norm{(a_h)_{ij} - (a_0)_{ij}}_{L^\infty(B)} \right) \\
&\hspace{4em} = \norm{A_h - A_0}_{(L^{d/2s}(B))^{d\times d}} \norm{A_h - A_0}_{(L^\infty(B))^{d\times d}}.
\end{align*}
Observing that $\norm{A_h - A_0}_{(L^\infty(B))^{d\times d}}$ is uniformly bounded as $h\to0$ due to convergence of $\{A_h\}$ to $A_0$ in $W_\Sigma^{1,\infty}(B)$, it follows that \eqref{rem1} holds. We now apply the same reasoning to establish \eqref{rem2}, and we begin by considering the exponent $q = \frac{p' p''}{2p'' - p'}$, which may be explicitly written as
\begin{equation*}
q = \left\{ \begin{array}{cl} \frac{1+\epsilon}{1-\epsilon} &\text{ if } d=2, \\ \frac{3}{2} &\text{ if } d=3. \end{array} \right.
\end{equation*}
Noting that $p'' < p' < \frac{p' p''}{2p'' - p'}$ (for either choice of the dimension $d$), we apply H{\"o}lder's interpolation inequality with $p = p''$, $q = \frac{p' p''}{2p'' - p'}$, $r = p'$, and $\alpha = \frac{1}{2}$ to obtain
\begin{equation*}
\norm{n_h - n_0}_{L^{p'}(B)}^2 \le \norm{n_h - n_0}_{L^{p''}(B)} \norm{n_h - n_0}_{L^q(B)}.
\end{equation*}
Convergence of $\{n_h\}$ to $n_0$ in $L^2(B)$ implies uniform boundedness of $n_h - n_0$ in the weaker $L^q(B)$-norm as $h\to0$, from which we conclude that \eqref{rem2} is satisfied. As a final note, we may absorb the denominator of the correction term into the constant since it is independent of $h$, providing the estimate \eqref{eig_est}. \qed

\end{proof}

We have considered perturbations that only require convergence of the coefficients in certain spaces, but of course we may restrict to specific types of perturbations (such as those corresponding to $L^\infty(B)$-bounded diametrically small perturbations studied in \cite{cakoni_moskow,cakoni_moskow_rome}) in order to gain more information about the sensitivity of the eigenvalues in those cases. We will provide such an example in a computational setting in Section \ref{sec:numerics} by examining the effect of introducing a circular void to the medium.

\section{The isotropic case}
\label{sec:isotropic}

While our main interest has concerned the effects of perturbations of an anisotropic medium on the set of $\delta$-Stekloff eigenvalues, it is worth investigating the isotropic case (in which $A = I$) to obtain improved bounds that result from the higher regularity that is available in this case. Thus, we assume for the remainder of this section that $A_h = A_0 = I$. The following result is an analogue of Proposition \ref{prop:A} and follows directly from standard elliptic regularity estimates (cf. \cite{adams}).

\begin{proposition} \label{prop:iso}

If $w\in H^1(B)$ satisfies
\begin{subequations} \label{w}
\begin{align}
\Delta w - k^2 w &= f \text{ in } B, \\
\normal{w} &= \xi \text{ on } \partial B,
\end{align}
\end{subequations}
for given $f\in L^2(B)$ and $\xi\in H^{1/2}(\partial B)$, then $w_h\in H^2(B)$ with the estimate
\begin{equation} \label{w_est}
\norm{w}_{H^2(B)} \le C \left( \norm{f}_{L^2(B)} + \norm{\xi}_{H^{1/2}(\partial B)} \right).
\end{equation}

\end{proposition}

We continue with the following analogue of Lemma \ref{lemma:Th_norm}, which asserts the norm convergence of $\{T_h(\lambda)\}$ to $T_0(\lambda)$, independently of $\lambda\in U$, with an improved assumption on the convergence of $\{n_h\}$. We remark that this improvement is a result of no longer needing to uniformly bound the norm $\norm{n_h u}_{\tilde{H}^{s-1}(B)}$, which required convergence of $\{n_h\}$ in $L^2(B)$ (more specifically, in $L^{3/(2-s)}(B)$). Otherwise, the proof is identical to the estimate of the second term in the proof of Lemma \ref{lemma:Th_norm}.

\begin{lemma} \label{lemma:Th_norm_iso}

If $n_h\to n_0$ in $L^{3/2}(B)$ as $h\to0$, then the sequence $\{T_h(\lambda)\}$ satisfies the norm estimate
\begin{equation} \label{Th_norm_iso}
\norm{T_h(\lambda) - T_0(\lambda)} \le C \norm{n_h - n_0}_{L^{p_0}(B)},
\end{equation}
where for each $\epsilon>0$ we have
\begin{equation*}
p_0 = \left\{\begin{array}{cl} 1+\epsilon &\text{ if } d=2, \\ \frac{3}{2} &\text{ if } d=3, \end{array} \right.
\end{equation*}
and the constant $C$ is independent of $h$ and $\lambda\in U = \mathbb{C}\setminus\overline{B_R(0)}$ but depends on $\epsilon$ whenever $d=2$. As a result, we have $T_h(\lambda)\to T_0(\lambda)$ in norm as $h\to0$, uniformly for $\lambda\in U$.

\end{lemma}

With these initial results in hand, we may state the improved version of Theorem \ref{theorem:main} in the isotropic case, noting that for $d=3$ the new exponent $p'$ may be obtained formally by setting $s = \frac{1}{2}$ in \eqref{pprime}.

\begin{theorem} \label{theorem:main_iso}

Suppose that $n_h\to n_0$ in $L^{3/2}(B)$ as $h\to0$, and suppose that there exist no nontrivial solutions of \eqref{dir} for $(A,n) = (I,n_0)$. Let $\lambda_0$ be a nonzero simple $\delta$-Stekloff eigenvalue for $(A,n) = (I,n_0)$ with $H^1(B)$-normalized eigenfunction $u_0$, choose $R>0$ such that $R<\abs{\lambda_0}$, and let $U = \mathbb{C}\setminus\overline{B_R(0)}$. For sufficiently small $h>0$ there exists a simple $\delta$-Stekloff eigenvalue for $(A,n) = (I,n_h)$ that satisfies the formula
\begin{equation} \label{Th_main_iso}
\lambda_h = \lambda_0 + \frac{k^2 ((n_h-n_0)u_0,u_0)_B}{\inner{S_\delta u_0}{u_0}_{\partial B}} + O\left( \norm{n_h-n_0}_{L^{p'}(B)}^2 \right),
\end{equation}
where the exponent $p'$ is given by
\begin{equation} \label{pprime_iso}
p' = \left\{ \begin{array}{cl} 1+\epsilon &\text{ if } d=2, \\ \frac{6}{5} &\text{ if } d=3, \end{array} \right.
\end{equation}
for a given $\epsilon>0$.

\end{theorem}

\begin{proof} From Lemma \ref{lemma:Th_norm_iso} we may apply Theorem \ref{theorem:moskow} as we did in the proof of Theorem \ref{theorem:main} for the anisotropic case. The correction term in the present case is obtained by setting $A_h = A_0 = I$ in \eqref{Th_main}, and as a consequence we need only bound the remainder term. For any $\lambda\in U$ we observe that for all $u\in R(E_{\lambda_0})$ and $u'\in H^1(B)$ we have
\begin{equation*}
\abs{((T_h(\lambda) - T_0(\lambda))u,u')_B} \le k^2 \abs{\lambda}^{-1} \abs{((n_h - n_0)u,u')_B}.
\end{equation*}
By Proposition \ref{prop:iso} we have $u\in H^2(B)$ with $\norm{u}_{H^2(B)} \le C\norm{u}_{H^1(B)}$, and the Sobolev embedding theorem implies that $u\in C_b(B)$ for $d=2,3$. Thus, we may apply the three-term H{\"o}lder's inequality with $p = p'$, $q = \infty$, and $r\in[2,\infty)$ when $d=2$ and $r=6$ when $d=3$ to obtain
\begin{equation*}
\abs{((T_h(\lambda) - T_0(\lambda))u,u')_B} \le C\norm{n_h - n_0}_{L^{p'}(B)} \norm{u}_{H^1(B)} \norm{u'}_{H^1(B)}.
\end{equation*}
The definition of the operator norm yields
\begin{equation*}
\norm{(T_h(\lambda) - T_0(\lambda))|_{R(E_{\lambda_0})}} \le C\norm{n_h - n_0}_{L^{p'}(B)}
\end{equation*}
with $C$ independent of $\lambda\in U$. By the same reasoning as in the proof of Theorem \ref{theorem:main}, the adjoint operators restricted to $R(E_{\lambda_0}^*)$ satisfy the same bound, and the result follows. \qed

\end{proof}

\begin{remark} \label{remark:iso}

In \cite{cakoni_colton_meng_monk}, perturbations of Stekloff eigenvalues (i.e., $\delta=0$) corresponding to an isotropic medium with $n$ real-valued were shown to satisfy the approximate formula (recast in the notation of the present work)
\begin{equation*}
\lambda_h - \lambda_0 \approx \frac{k^2((n_h - n_0)u_0,u_0)_B}{\inner{u_0}{u_0}_{\partial B}},
\end{equation*}
where second-order terms have been ignored. We remark that \eqref{Th_main_iso} provides a rigorous form of this approximation with a remainder term that is also valid for generally complex-valued $n$. The same approach may be used to rigorously justify a similar approximate formula in \cite{cogar_colton_meng_monk} for the related class of modified transmission eigenvalues.

\end{remark}

We end this section by presenting the analogue of Corollary \ref{corollary:eig_est} in the isotropic case, which is an improvement in the case $d=3$ and may be obtained formally by setting $s = \frac{1}{2}$ in \eqref{pdoubleprime}.

\begin{corollary} \label{corollary:eig_est_iso}

Under the hypotheses of Theorem \ref{theorem:main_iso}, for sufficiently small $h>0$ we have the bound
\begin{equation} \label{eig_est}
\abs{\lambda_h - \lambda_0} \le C \norm{n_h-n_0}_{L^1(B)}.
\end{equation}

\end{corollary}

\begin{proof} We observe that
\begin{equation*}
\abs{((T_h(\lambda_0) - T_0(\lambda_0))u_0,u')_B} \le k^2 \abs{\lambda_0}^{-1} \abs{((n_h - n_0)u_0,u_0)_B}.
\end{equation*}
As in the proof of Theorem \ref{theorem:main_iso}, from $u_0\in H^2(B)$ the Sobolev embedding theorem implies that $u_0\in C_b(B)$ for $d=2,3$ with continuous embedding. Noting that $u_0$ now appears in both components of the inner product, we apply the three-term H{\"o}lder's inequality with $p = 1$ and $q=r=\infty$ to obtain
\begin{equation*}
\abs{((T_h(\lambda_0) - T_0(\lambda_0))u_0,u')_B} \le C\norm{n_h - n_0}_{L^1(B)}.
\end{equation*}
All that remains is to verify that the remainder term may be absorbed into the right-hand side of this estimate, as we did in the proof of Corollary \ref{corollary:eig_est}. We begin with the observation that the exponent $q = \frac{p'}{2-p'}$ may be written explicitly as
\begin{equation*}
q = \left\{ \begin{array}{cl} \frac{1+\epsilon}{1-\epsilon} &\text{ if } d=2, \\ \frac{3}{2} &\text{ if } d=3. \end{array} \right.
\end{equation*}
Noting that $1<p'<q$, we apply H{\"o}lder's interpolation inequality with $p = 1$, $q = \frac{p'}{2-p'}$, $r=p'$, and $\alpha = \frac{1}{2}$ to obtain
\begin{equation*}
\norm{n_h - n_0}_{L^{p'}(B)}^2 \le \norm{n_h - n_0}_{L^1(B)} \norm{n_h - n_0}_{L^q(B)}.
\end{equation*}
Convergence of $\{n_h\}$ to $n_0$ in $L^{3/2}(B)$ and hence also in the (possibly) weaker $L^q(B)$-norm implies that $\norm{n_h-n_0}_{L^q(B)}$ is uniformly bounded as $h\to0$. Thus, we conclude that \eqref{eig_est} holds. \qed

\end{proof}

\section{Numerical examples}
\label{sec:numerics}

We now investigate the sensitivity of $\delta$-Stekloff eigenvalues corresponding to an isotropic medium through an example in two dimensions. We refer to \cite{cogar2020} for a description of how $\delta$-Stekloff eigenvalues and the corresponding eigenfunctions may be computed using the finite element method. We choose a medium such that $n_0=4$ in an L-shaped domain $D$ (depicted in Figure \ref{fig:ef}) given by removing the square $[0.1,1.1]\times[-1.1,-0.1]$ from the square $[-0.9,1.1]\times[-1.1,0.9]$ and $n_0=1$ elsewhere, and we choose the auxiliary domain $B$ to be the disk of radius $1.5$ centered at the origin, which guarantees that the closure of $D$ is contained within $B$. Finally, we choose $k=1$ and $\delta = \frac{1}{2}$. \par

For the perturbed medium, we introduce a circular void of radius $h>0$ centered at $(x_c,y_c) = (0.1,0.4)$ that we denote by $B_h := B_h(x_c,y_c)$, in which case we may write $n_h$ as
\begin{equation*}
n_h = \left\{ \begin{array}{cl} 1 &\text{ in } B_h, \\ n_0 &\text{ elsewhere}. \end{array} \right.
\end{equation*}
We note that $\text{supp}(n_h - n_0) = B_h$ and $\norm{n_h - n_0}_{L^\infty(B)} = \abs{n_0 - 1} = 3$, from which we obtain
\begin{equation*}
\norm{n_h - n_0}_{L^{3/2}(B)} \le \norm{n_h - n_0}_{L^\infty(B)} \left( \int_{B_h} \,dx \right)^{2/3} = 3\pi h^{4/3}.
\end{equation*}
In particular, we observe that $n_h\to n_0$ in $L^{3/2}(B)$, and Theorem \ref{theorem:main_iso} asserts that for sufficiently small $h$ we have the eigenvalue bound
\begin{equation} \label{eigbound}
\abs{\lambda_h - \lambda_0} \le C \norm{n_h - n_0}_{L^1(B)} \le C h^2.
\end{equation}
We recall from \eqref{Th_main_iso} that $\lambda_h - \lambda_0$ is also related to the magnitude of the eigenfunction $u_0$ in a neighborhood of the support of $n_h - n_0$, motivating us to investigate the eigenvalue bound for two eigenvalues whose respective eigenfunctions have different magnitudes near $B_h$. The normalized eigenfunctions for two such eigenvalues are shown in Figure \ref{fig:ef}, where we have depicted the boundary of the region $B_h$ for $h=0.1$ as a white circle. The corresponding shifts in the two eigenvalues when the radius $h$ is reduced logarithmically from $h=0.1$ to $h=0.01$ are shown in Figure \ref{fig:eigs}. 

\begin{figure}[htp]
\begin{subfigure}{.5\textwidth}
\centering
\includegraphics[width=1\linewidth]{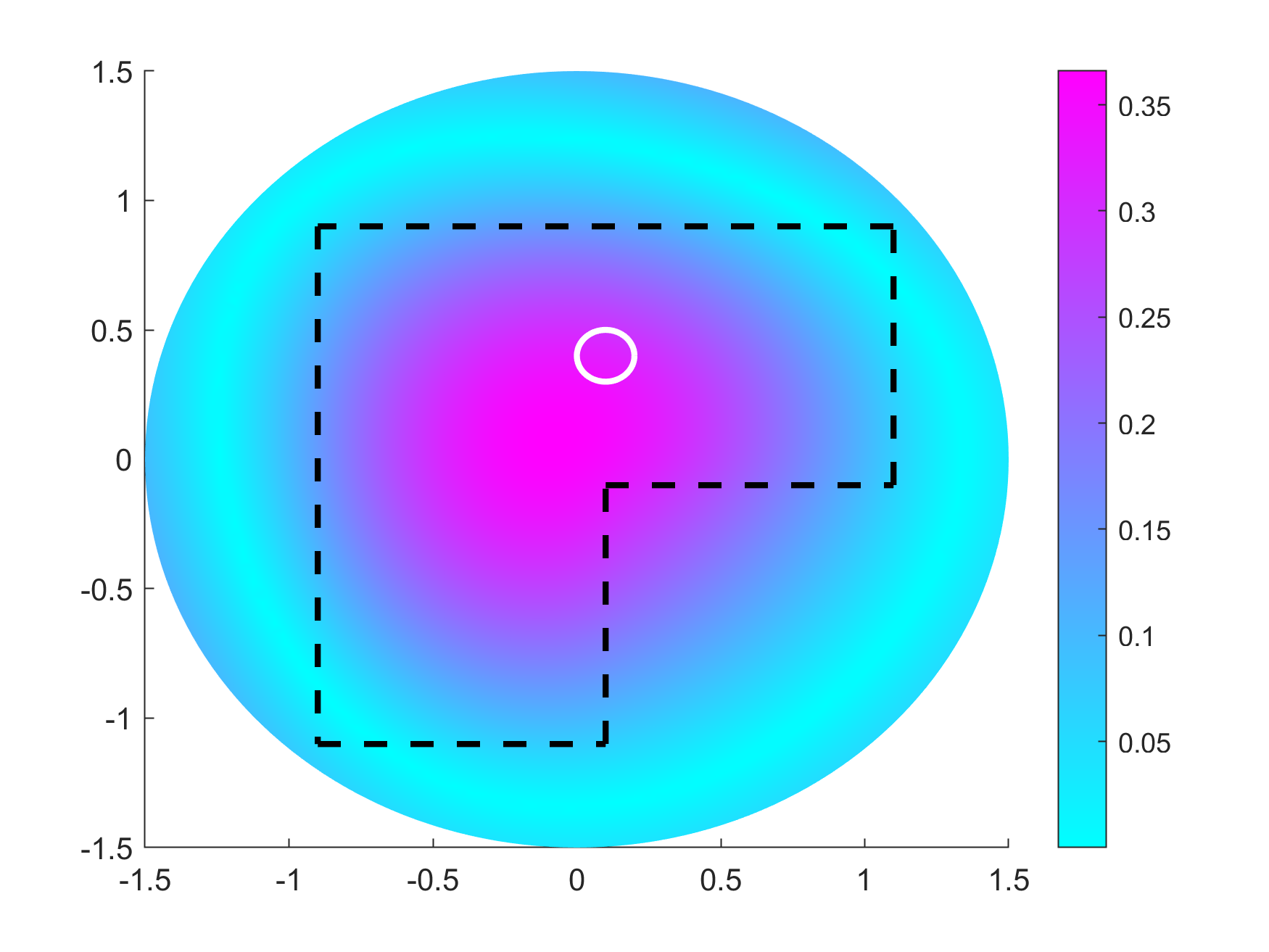}
\caption{$\lambda_0 = -4.25$}
\label{fig:ef1_Lshape_deltahalf}
\end{subfigure}
\begin{subfigure}{.5\textwidth}
\centering
\includegraphics[width=1\linewidth]{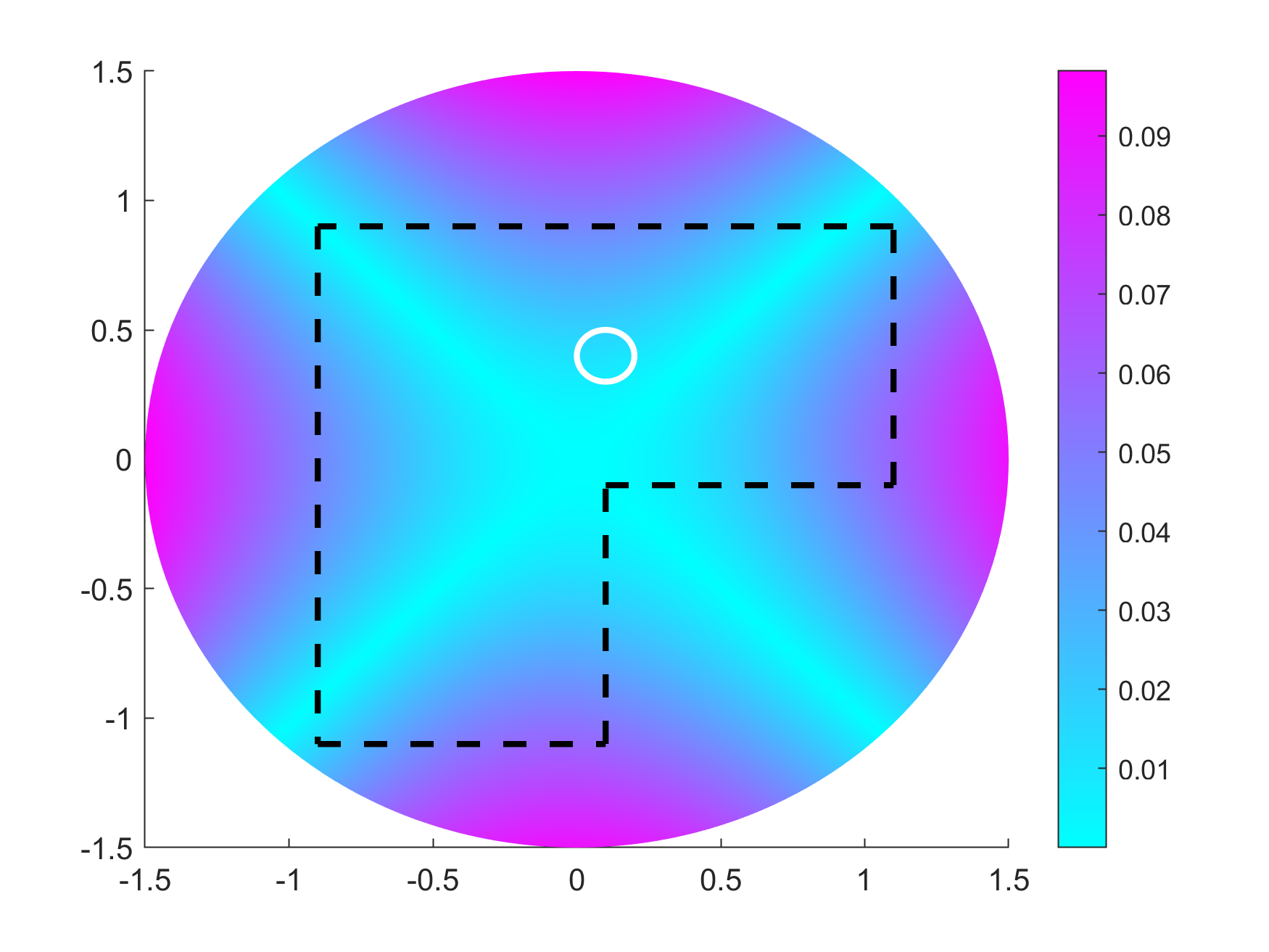}
\caption{$\lambda_0 = -2.16$}
\label{fig:ef2_Lshape_deltahalf}
\end{subfigure}
\caption{Plots of the magnitudes of two eigenfunctions corresponding to $\delta$-Stekloff eigenvalues for an L-shaped domain. The eigenfunction corresponding to $\lambda=-4.25$ is much larger than that corresponding to $\lambda=-2.16$ in a neighborhood of $B_h$, which is outlined in white for $h = 0.1$.}
\label{fig:ef}
\end{figure}

\begin{figure}[htp]
\centering
\includegraphics[width=\linewidth]{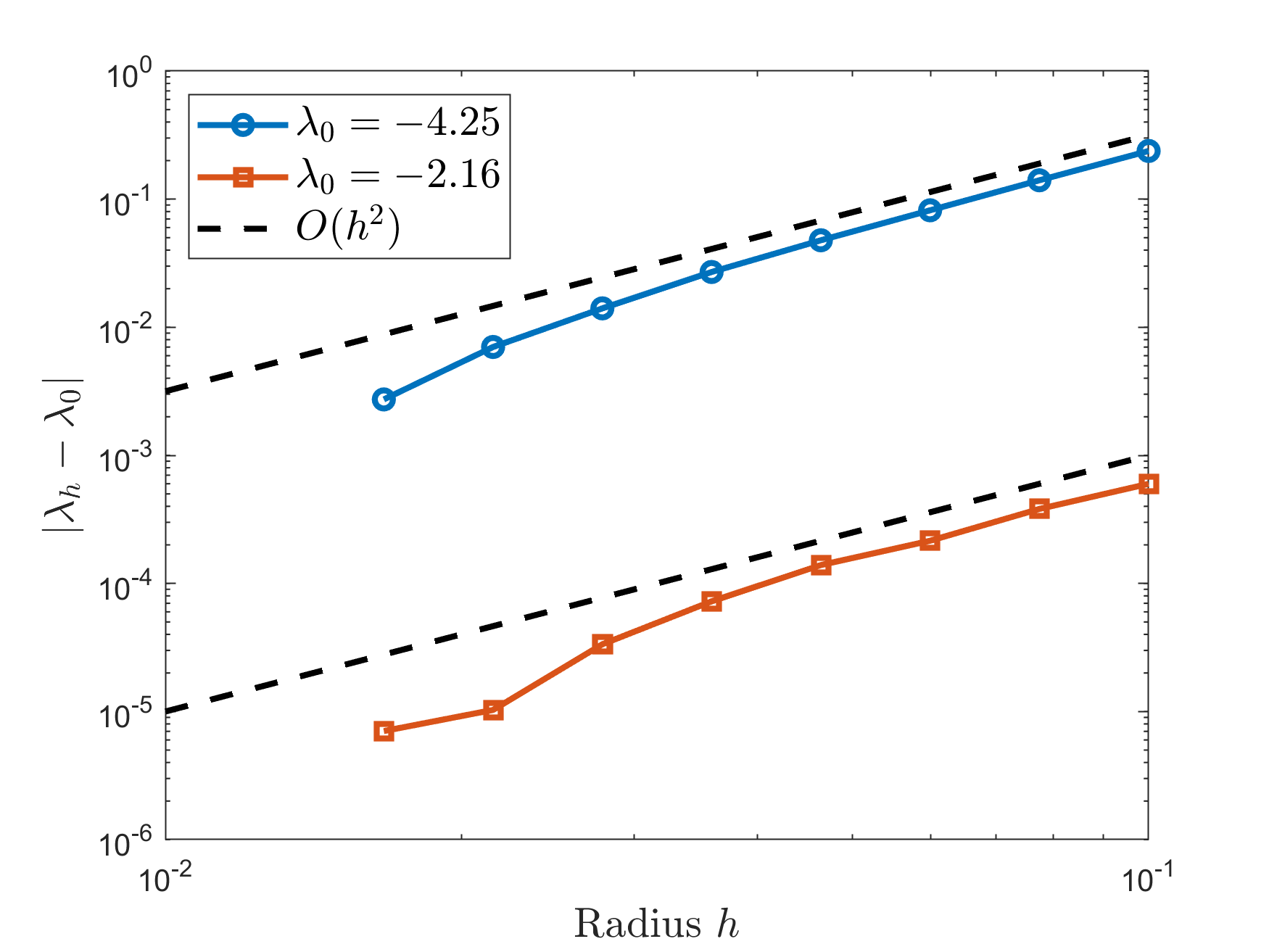}
\caption{Magnitudes of the shifts of the eigenvalues $\lambda_0 = -4.25$ (blue circles) and $\lambda_0 = -2.16$ (red squares). The black dashed lines represent $O(h^2)$-convergence.}
\label{fig:eigs}
\end{figure}

As expected from the observation that the eigenfunction for $\lambda_0 = -4.25$ has significantly higher magnitude near $B_h$ than the eigenfunction for $\lambda_0 = -2.16$, we see in Figure \ref{fig:eigs} that the former eigenvalue shifts to a greater degree due to changes in $n$. Moreover, by comparing with the black dashed lines indicating $O(h^2)$-convergence, we confirm the bound \eqref{eigbound}; however, for smaller values of $h$ it appears that convergence might be even faster than $O(h^2)$, possibly influenced by the shape of the eigenfunctions in that region. This observation suggests that it might be of interest to study the structure of the eigenfunctions in greater detail, such as the location of nodal lines, the number of connected components, and where the eigenfunctions are concentrated. Notably, an eigenvalue whose eigenfunction is concentrated in the center of the domain $B$ (as seen in Figure \ref{fig:ef1_Lshape_deltahalf}) should display high sensitivity for most localized perturbations of $n$, and it would be advantageous to be able to identify such eigenvalues in a practical setting. 

\section{Conclusion}
\label{sec:conclusion}

By reformulating the $\delta$-Stekloff eigenvalue problem as a nonlinear eigenvalue problem, we have derived precise first-order asymptotic correction formulas for perturbations of the eigenvalues due to changes in the coefficients $A$ and $n$ of the medium. In particular, we found that the regularity of $A$ has a strong effect on the rate of convergence of the eigenvalues as $A_h\to A_0$ and $n_h\to n_0$ in appropriate norms, with the general effect of higher regularity increasing this rate. We also investigated--both theoretically and with a simple numerical example--the sensitivity of the eigenvalues in the isotropic case, which allowed us to leverage classical elliptic regularity estimates to obtain improved bounds. While the numerical example was restricted to perturbations of a piecewise-homogeneous isotropic medium, it serves as an initial step toward the study of more complicated examples involving a possibly heterogeneous absorbing medium. \par
A possible application of the asymptotic formulas we have obtained is to go beyond simply detecting the presence of a flaw in a material and attempt to localize it, but the author is not aware of any results in this direction, even among a restricted class of perturbations. If we ignore the remainder term in the asymptotic formula for an isotropic medium, we are presented with the inverse problem of determining the support of the perturbation $n_h - n_0$ from a knowledge of $\lambda_h - \lambda_0$ for multiple eigenvalues and possibly their eigenfunctions. A constructive solution to this problem would be of immense value to the applicability of eigenvalues as target signatures in nondestructive testing. \par
Other interesting questions remain unanswered, one of which concerns the effect of the smoothing parameter $\delta$ on the sensitivity of the eigenvalues. This relationship is not clear from the asymptotic formulas that we have derived, as the eigenfunction $u_0$ also depends on $\delta$. A further modification of the $\delta$-Stekloff eigenvalue problem was investigated in \cite{cogar2020}, in which the lowest order Fourier coefficient in the explicit representation of $S_\delta$ was increased by an amplification factor $\sigma$, and a series of numerical examples showed an improvement in sensitivity of one of the eigenvalues. It would be of interest to examine this observation in the theoretical context that we have developed in the preceding sections. \par
Finally, the class of $\delta$-Stekloff eigenvalues was extended to the case of electromagnetic scattering in \cite{cogar2020EM}. Similar to \cite{cogar2020}, stability results were obtained for the eigenvalues but no quantitative formulas were derived. Given the similar structure in this case to the present problem, it may be possible to apply the same techniques to arrive at asymptotic formulas for the eigenvalues; however, the more complicated solvability requirements for electromagnetic problems (namely the strong dependence of compactness results on the coefficients) leaves it unclear how to reformulate the eigenvalue problem as a nonlinear eigenvalue problem that satisfies the hypotheses of Theorem \ref{theorem:moskow}. \\




%
%

\bibliographystyle{spmpsci}      

\end{document}